\documentclass[10pt]{article}
\usepackage[a4paper]{geometry}
\usepackage[utf8]{inputenc}
\usepackage[T1]{fontenc}
\usepackage[english]{babel}
\usepackage[babel=true,kerning=true]{microtype}
\usepackage{lmodern}

\usepackage{pifont}
\rmfamily
\DeclareFontShape{T1}{lmr}{b}{sc}{<->ssub*cmr/bx/sc}{}
\DeclareFontShape{T1}{lmr}{bx}{sc}{<->ssub*cmr/bx/sc}{}

\usepackage{tikz}
\usepackage{wrapfig} 
\usepackage[all]{xy}

\usetikzlibrary{shapes,patterns,plotmarks}
\usetikzlibrary{arrows}
\usepackage{amsfonts,amssymb}
\usepackage{amsmath,amsthm,amscd}

\usepackage{bbm}

\usepackage{array}
\usepackage{epsfig,graphics,graphicx,xcolor}
\usepackage{stmaryrd}

\usepackage[babel=true]{csquotes}
\usepackage{enumitem}
\usepackage{mathtools}
\usepackage[numbers]{natbib}

\DeclarePairedDelimiter\abs{\lvert}{\rvert}%

\DeclareFontFamily{U}{mathx}{\hyphenchar\font45}
\DeclareFontShape{U}{mathx}{m}{n}{
      <5> <6> <7> <8> <9> <10>
      <10.95> <12> <14.4> <17.28> <20.74> <24.88>
      mathx10
      }{}
\DeclareSymbolFont{mathx}{U}{mathx}{m}{n}
\DeclareMathSymbol{\bigtimes}{1}{mathx}{"91}

  \theoremstyle{plain}
\newtheorem{theorem}{Theorem}[section]
\newtheorem{proposition}[theorem]{Proposition}
\newtheorem{corollary}[theorem]{Corollary}
\newtheorem{lemma}[theorem]{Lemma} 
\newtheorem{deflemma}[theorem]{Definition-Lemma} 
  \theoremstyle{remark}
\newtheorem{remark}[theorem]{Remark}
  \theoremstyle{definition}
\newtheorem{definition}[theorem]{Definition}

\newcommand{\resp}{\textit{resp. }}
\newcommand{\ie}{\textit{i.e. }}

\newcommand{\NN}{\mathbb{N}}
\newcommand{\ZZ}{\mathbb{Z}}
\newcommand{\RR}{\mathbb{R}}
\renewcommand{\SS}{\mathbb{S}}
\newcommand{\FF}{\mathbb{F}}
\newcommand{\QQ}{\mathbb{Q}}
\newcommand{\CC}{\mathbb{C}}

\renewcommand{\S}{\operatorname{S}}
\newcommand{\T}{\operatorname{T}}

\newcommand{\ab}{\operatorname{ab}}
\newcommand{\nab}{\raisebox{-.65ex}{\ensuremath{\nabla}}}
\newcommand{\RW}{\raisebox{-.75ex}{\ensuremath{\mathrm{R,\!W}}}}

\newcommand{\PW}
{\raisebox{-.75ex}{\ensuremath{\mathrm{P,\!W}}}}

\newcommand{\AW}
{\raisebox{-.75ex}{\ensuremath{\!\mathrm{AW}}}}

\newcommand{\AWnab}
{\raisebox{-.75ex}{\ensuremath{\!\mathrm{AW},\!\nabla}}}

\title{Virtual knot theory on a group}
\author{Arnaud Mortier \\ \itshape arno.mortier@laposte.net}
\date{\today}

\begin{document}

\maketitle

\begin{abstract}
\footnotesize

Given a group endowed with a Z/2-valued morphism we associate a Gauss diagram theory, and show that for a particular choice of the group these diagrams encode faithfully virtual knots on a given arbitrary surface.
This theory contains all of the earlier attempts to decorate Gauss diagrams, in a way that is made precise via \textit{symmetry-preserving maps}. These maps become crucial when one makes use of decorated Gauss diagrams to describe finite-type invariants. In particular they allow us to generalize Grishanov-Vassiliev's formulas and to show that they define invariants of virtual knots.

\end{abstract}

\tableofcontents

\vspace*{0.2cm}

Gauss diagrams were introduced in knot theory as a means of representing knots and their finite-type invariants \citep{PolyakViro, Goussarov}, allowing compactification and generalization of formulas due to J.Lannes \citep{Lannes}. Since then, several generalizations have been attempted to adapt them to knot theory in thickened surfaces by decorating them with topological information \citep{Fiedler, Grishanov, MortierPolyakEquations}.

Our goal is to construct a unifying \enquote{father} framework, and to describe how to get down from there to other versions with less data.

First we define and study (virtual) knot diagrams on an arbitrary surface $\Sigma$: these are tetravalent graphs embedded in $\Sigma$, some of whose double points (the \enquote{real} ones) are pushed and desingularized into a real line bundle over $\Sigma$. Defining Gauss diagrams requires a global notion for the branches at a real crossing to be one \enquote{over} the other, and a global notion of writhe of a crossing. It is shown that these notions can be defined simultaneously if and only if $\Sigma$ is orientable. If it is not, we sacrifice the globality of one property, and take into account its monodromy. It is shown that when the total space of the bundle is orientable, the writhes are globally defined and the monodromy of the \enquote{over/under} datum is the first Stiefel-Whitney class of the tangent bundle to $\Sigma$, $w_1(\Sigma)$.

In Section~\ref{sec:VKTG} is given a definition of Gauss diagrams decorated by elements of a fixed group $\pi$, subject to usual Reidemeister moves, and to additional  \enquote{conjugacy moves}, depending on a fixed group homomorphism $w:\pi\rightarrow \FF_2$. It is shown that when there is a surface $\Sigma$ such that $\pi=\pi(\Sigma)$ and $w_1=w_1(\Sigma)$, then there is a $1-1$ correspondence between Gauss diagrams and virtual knot diagrams, that induces a correspondence between the equivalence classes (virtual knot types) on both sides.

A lighter kind of Gauss diagrams, called \textit{abelian}, is defined in Subsection~\ref{subsec:abelian} following the idea of T.Fiedler's $H_1(\Sigma)$-decorated diagrams (\cite{Fiedler}) and shown to be equivalent to the above when $\pi$ is abelian and $w$ is trivial. The little drawback of this version is that it becomes more difficult to compute the homological decoration of an arbitrary loop. Two formulas are presented in \ref{subsec:homform} to sort this out, involving quite unexpected combinatorial tools.

Finally, we describe invariance criteria for the analog of Goussarov-Polyak-Viro's invariants \citep{GPV} in this framework. As an application, we obtain a generalization of Grishanov-Vassiliev formulas \citep{Grishanov}, and a notion of Whitney index for virtual knots whose underlying immersed curve is non nullhomotopic.

\subsection*{Acknowledgements}
I wish to thank Christian Blanchet who invited me at the Institut de Math\'ematiques de Jussieu where this work was done. This work has benefited from discussions with Michael Polyak, Micah Chrisman, and Thomas Fiedler. 

\section{Preliminary: classical Gauss diagrams and their Reidemeister moves}\label{subsec:classical}

\begin{definition}A \textit{classical Gauss diagram} is an equivalence class of an oriented circle in which a finite number of couples of points are linked by an abstract oriented arrow with a sign decoration, up to positive homeomorphism of the circle.  A Gauss diagram with $n$ arrows is said to be of \textit{degree} $n$.
\end{definition}
It may happen that one regards Gauss diagrams as topological objects (drawing loops on them, considering their first homology). In that case, one must beware of the fact that \textit{the arrows do not topologically intersect} -- that is what is meant by \enquote{abstract}. However, the fact that two arrows may \textit{look like} they intersect is something combinatorially well-defined, and interesting for many purposes.

\textbf{Fact:}
There is a natural way to associate a Gauss diagram with a knot diagram in the sphere $\SS^2$, from which the knot diagram can be uniquely recovered.  Fig.\ref{1} illustrates this fact.

\begin{figure}[h]
\centering 
\psfig{file=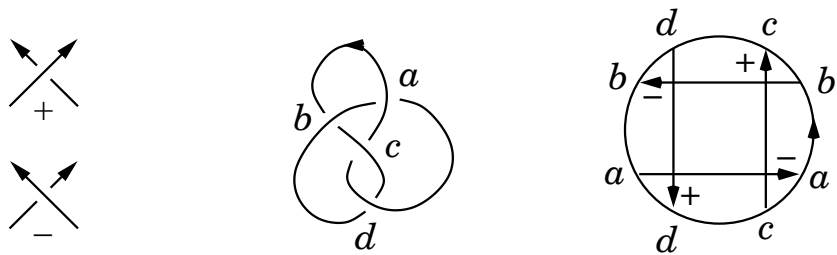,scale=0.9}
\caption{The writhe convention, a diagram of the figure eight knot, and its Gauss diagram -- the letters are here only for the sake of clarity.}\label{1}
\end{figure}

However, not every Gauss diagram actually comes from a knot diagram in that way. This observation has lead to the development of virtual knot theory \citep{KauffmanVKT99}: basically a virtual knot is a Gauss diagram which does not come from an actual knot. There is a knot-diagrammatic version of these, using virtual crossings subject to virtual Reidemeister moves - that can be thought of as a unique \enquote{detour move}. A detour move is naturally any move that leaves the underlying Gauss diagram unchanged.

Of course virtual knot diagrams are also subject to the usual Reidemeister moves, and these do change the face of the Gauss diagram. 
We call them R-moves for simplicity - and to make it clear whether knot diagrams or Gauss diagrams are considered. Here is a combinatorial description of R-moves.

\subsubsection*{R$_1$-moves}

An R$_1$-move is the birth or death of an isolated arrow, as shown in Fig.\ref{Rmoves} (top-left). There is no restriction on the direction or the sign of the arrow.

\subsubsection*{R$_2$-moves}

An R$_2$-move is the birth or death of a pair of arrows with different signs, whose heads are consecutive as well as their tails (Fig.\ref{Rmoves}, top-right).

If one restricts oneself to Gauss diagrams that come from classical knot diagrams, then there is an additional condition as for the creating direction: indeed, two arcs in a knot diagram can be subject to a Reidemeister II move if and only if they \textit{face each other}. In the virtual world, there is no such condition since any two arcs can be brought to face each other by detour moves.

It may be good to know that this condition can be read directly on the Gauss diagram: indeed, two arcs face each other in a knot diagram if one can join them by walking along the diagram and \textit{turning to the left} at each time one meets a crossing. Thanks to the decorations of the arrows, it makes sense for a path in a Gauss diagram to \textit{turn to the left}.

\begin{figure}[h!]
\centering 
\psfig{file=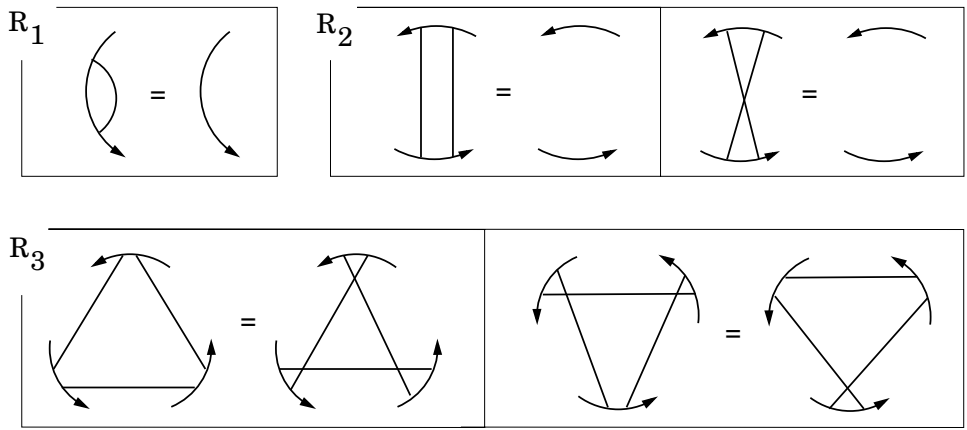,scale=0.7}
\caption{$\textrm{R}$-moves for Gauss diagrams (see above and below the rules for the decorations)}\label{Rmoves}
\end{figure}

\subsubsection*{R$_3$-moves}

\begin{definition}\label{def_epsilon}
In a classical Gauss diagram of degree $n$, the complementary of the arrows is made of $2n$ oriented components. These are called the \textit{edges} of the diagram. In a diagram with no arrow, we still call the whole circle an edge.

Let $e$ be an edge in a Gauss diagram, between two consecutive arrow ends that do not belong to the same arrow.
Put
$$\eta (e)=\left\lbrace \begin{array}{l}
+1 \text{ if the arrows that bound $e$ cross each other} \\
-1 \text{ otherwise}
\end{array}\right. ,$$
and let $\uparrow\!\!(e)$ be the number of arrowheads at the boundary of $e$.
Then define
$$\varepsilon (e)= \eta(e)\cdot(-1)^{\uparrow(e)}.$$
Finally, define $\mathrm{w}(e)$ as the product of the writhes of the two arrows at the boundary of $e$.

\end{definition}

An R$_3$-move is the simultaneous switch of the endpoints of three arrows as shown on Fig.\ref{Rmoves} (bottom), with the following conditions:
\begin{enumerate}
\item The value of $\mathrm{w}(e)\varepsilon(e)$ should be the same for all three visible edges $e$. This ensures that the piece of diagram containing the three arrows can be represented in a knot-diagrammatic way without making use of virtual crossings.
\item The values of  $\uparrow\!\!(e)$ should be pairwise different. This ensures that one of the arcs in the knot diagram version actually \enquote{goes over} the others.
\end{enumerate}

\begin{remark}
From a simplicial viewpoint, the sign $\mathrm{w}(e)\varepsilon(e)$ gives a natural co-orientation of the $1$-codimensional strata corresponding to R$_3$ moves. This is exploited in \citep{FT1cocycles} to construct finite-type $1$-cocycles.
\end{remark}

\section{Knot and virtual knot diagrams on an arbitrary surface}\label{sec:KDAS}

The goal of this section is to examine when and how one can define a couple of equivalent theories \enquote{virtual knots $-$ Gauss diagrams} that generalizes knot theory in an arbitrary $3$-manifold $M$. What first appears is that a Gauss diagram depends on a projection; so it seems unavoidable to ask for the existence of a surface $\Sigma$ (maybe with boundary, non orientable, or non compact), and a \enquote{nice} map $p:M\rightarrow \Sigma$. For the \textit{over} and \textit{under} branches at a crossing to be well-defined at least locally, the fibers of $p$ need to be equipped with a total order: this leaves only the possiblity of a real line bundle.

\subsection{Thickenings of surfaces}

Let us now split the discussion according to the two kinds of decorations that one would expect to find on a Gauss diagram: signs (local writhes), and orientation of the arrows.

\subsubsection*{Local writhes}

For a knot in an arbitrary real line bundle, there are situations in which it is possible to switch over and under in a crossing by a mere \textit{diagram isotopy}. For instance, in the non-trivial line bundle over the annulus $\SS^1\times\RR$, a full rotation of the closure of the two-stranded elementary braid $\sigma_1$ turns it into the closure of $\sigma_1^{-1}$ (Fig.~\ref{pic:nonorientable}).

\begin{figure}[h!]
\centering 
\psfig{file=
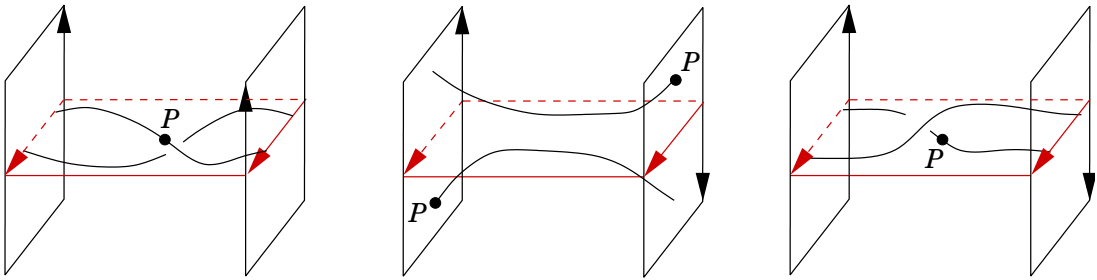,scale=1.3}
\caption{Non trivial line bundle over the annulus -- as one reads from left to right, the knot moves towards the right of the picture.}\label{pic:nonorientable}
\end{figure}

Fig.~\ref{pic:nonorientable} would be exactly the same (except for the gluing indications) if one considered the trivial line bundle over the Moebius strip. Note that this diagram would then represent a $2$-component link. In fact, it is possible to embed this picture in \textit{any non-orientable total space} of a line bundle over a surface.

This phenomenon reveals the fact that in these cases, there is no way to define the local writhe of a crossing. However, according to \cite{FiedlerSSS} (Definition~$1.$), there is a well-defined writhe as soon as the total space of the bundle is oientable. 

\begin{definition}\label{def:thick}
We call a \textit{thickened surface} a real line bundle over a surface, whose total space is orientable.
\end{definition}

\begin{deflemma}\label{def:wfg}
If $M\rightarrow\Sigma$ is a thickened surface, then its first Stiefel-Whitney class coincides with that of the tangent bundle to $\Sigma$. 
This class induces a homomorphism $w_1(\Sigma):\pi_1(\Sigma)\rightarrow\FF_2$.
The couple $(\pi_1(\Sigma), w_1(\Sigma))$ is called the \textit{weighted fundamental group} of $\Sigma$. Note that in particular the thickening of $\Sigma$ is the trivial bundle $\Sigma\times\RR$ if and only if $\Sigma$ is orientable.
\end{deflemma}

\subsubsection*{Arrow orientations}

Note that the writhe of a crossing for a knot in $M\rightarrow\Sigma$ depends only on one choice, that of an orientation for $M$. The important thing is that this choice is \textit{global}, so that it makes sense to compare the writhes of different crossings (they live in \enquote{the same} $\ZZ/2\ZZ$).

Similarly, for the orientation of the arrows in a Gauss diagram to simultaneously make sense, one needs a global definition of the over/under datum at the crossings; that is, the fibres of $M\rightarrow\Sigma$ should be simultaneously and consistently oriented. In other words, $M\rightarrow\Sigma$ should be the trivial line bundle.

According to our definition of a thickened surface, this happens only if the surface is orientable.\newline

\textit{So it seems that one has a choice to make, either restricting one's attention to orientable surfaces, or taking into account the monodromy of whatever is not globally defined. Additional \textit{conjugacy moves} will be needed when one defines Gauss diagrams. The convention to consider only fibre bundles with an orientable total space is arbitrary, its only use is to reduce the number of monodromy morphisms to $1$ instead of $2$.}

\subsubsection*{Virtual knot diagrams on an arbitrary surface}

Fix an arbitrary surface $\Sigma$ and denote its thickening by $M\rightarrow\Sigma$.

\begin{definition}\label{def:virtualknotdiags}
A \textit{virtual knot diagram} on $\Sigma$ is a generic immersion $\SS^1\rightarrow \Sigma$ whose every double point has been decorated
\begin{itemize}
\item either with the designation \enquote{virtual} (which is nothing but a name),
\item or with a way to desingularize it locally into $M$, up to local isotopy.
\end{itemize}
\end{definition}

These diagrams are subject to the usual Reidemeister moves, dictated by local isotopy in $M$, and to the virtual \enquote{detour} moves which are studied in the next section. As explained before, if one chooses an orientation for $M$, then the real crossings of such a diagram have a well-defined writhe.

\subsection{Diagram isotopies and detour moves}\label{subsec:globalmoves}

Here by \textit{knot diagram} we mean a virtual knot diagram on a fixed arbitrary surface $\Sigma$, as defined above. In this case a \textit{diagram isotopy}, usually briefly denoted by $H:\operatorname{Id}\rightarrow h$, is the datum of a diffeomorphism $h$ of $\Sigma$ together with an isotopy from $\operatorname{Id}_\Sigma$ to $h$. A \textit{detour move} is a boundary-fixing homotopy of an arc that, before and after the homotopy, goes through only virtual crossings (such an arc is called \textit{totally virtual}). Though both of these processes seem rather simple, it will be useful to understand how they interact.

\begin{lemma}\label{lem:iso_detour} A knot diagram obtained from another by a sequence of diagram isotopies alternating with detour moves may always be obtained by a single diagram isotopy followed by detour moves. 
\end{lemma}

\begin{proof}
It is enough to show that a detour move $d$ followed by a diagram isotopy $\operatorname{Id}\rightarrow h$ may be replaced with a diagram isotopy followed by a detour move (without changing the initial and final diagrams). The initial diagram is denoted by $D$.

Call $\alpha$ the totally virtual arc that is moved by the detour move. By definition, $d(\alpha)$ is boundary-fixing homotopic to $\alpha$, and is totally virtual too. Thus, $h\left( d\left( \alpha\right) \right) $ and $h(\alpha)$ are totally virtual and boundary-fixing homotopic to each other. Since $h\left(d\left(D\right) \right) $ and $h(D)$ differ only by these two arcs, it follows that there is a detour move taking $h(D)$ to $h\left( d\left( D\right) \right) $.
\end{proof}

Now an interesting question about diagram isotopies is when two of them lead to diagrams that are equivalent \textit{under detour moves}.  Here is a quite useful sufficient condition.

\begin{definition}\label{def:generalizedbraids}
Let $X$ and $Y$ be two finite subsets of $\Sigma$ with the same (positive) cardinality $n$. A \textit{generalized braid} in $\Sigma\times \left[ 0,1\right]$ based on the sets $X$ and $Y$ is an embedding $\beta$ of a disjoint union of segments, such that $\operatorname{Im}\beta\cap\left( \Sigma\times \left\lbrace t\right\rbrace\right) $ has cardinality $n$ for each $t$, coincides with $X$ at $t=0$ and with $Y$ at $t=1$.
\end{definition}

Let $D$ be a knot diagram and $H$ a diagram isotopy. Let $p_1\in P_1,\ldots ,p_n\in P_n$ denote little neighborhoods of the real crossings of $D$,
and set $\mathcal{P}=\cup P_i$. Then, $\coprod H(p_i,\cdot)$ defines a generalized braid $\prescript{H\!\!}{}{\beta}$ in $\Sigma\times \left[ 0,1\right] $ with $n$ strands based on the sets $\left\lbrace p_1,\ldots ,p_n \right\rbrace$ and $\left\lbrace h(p_1),\ldots ,h(p_n) \right\rbrace$. The strand of a braid $\beta$ that intersects $\Sigma\times \left\lbrace 0\right\rbrace$ at $p_i$ is denoted by $\beta_i$.

\begin{proposition}
\label{prop:braid_homotopy}
Let $D$ and $H$ be as above. Then, up to detour moves, $h(D)$ only depends on $D$ and the boundary fixing homotopy class of $\prescript{H\!\!}{}{\beta}$.
\end{proposition}
\begin{proof}
Let $\gamma$ be a maximal smooth arc of $D$ outside $\mathcal{P}$ (thus totally virtual). It begins at some $P_i$ and ends at some $P_j$ (of course it may happen that $j=i$). Using little arcs inside of $P_i$ and $P_j$ to join the endpoints of $\gamma$ with $p_i$ and $p_j$, one obtains an oriented path $\prescript{H\!\!}{}{\beta}_i^{-1}\gamma\prescript{H\!\!}{}{\beta}_j$. 

The obvious retraction of $\Sigma\times\left[ 0,1\right]$ onto $\Sigma\times\left\lbrace 1\right\rbrace$ induces a map
$$\pi_1(\Sigma\times\left[ 0,1\right],h(\mathcal{P})\times\left\lbrace 1\right\rbrace)\longrightarrow \pi_1(\Sigma,h(\mathcal{P}))$$
that sends the class $\left[ \prescript{H\!\!}{}{\beta}_i^{-1}\gamma\prescript{H\!\!}{}{\beta}_j\right]$
to $\left[ h(\gamma)\right]$. Since the former class is unchanged under boundary-fixing homotopy of $\gamma$ and $\prescript{H\!\!}{}{\beta}$, so is the latter, which proves the result.
\end{proof}

This proposition states that the only relevant datum in a diagram isotopy of a virtual knot is the path followed by the real crossings along the isotopy, \textit{up to homotopy}: the entanglement of these paths with each other or themselves does not matter. It follows that the crossings may be moved one at a time:

\begin{corollary}\label{1by1}
Let $D$ be a knot diagram with its real crossings numbered from $1$ to $n$, and let $H:\operatorname{Id}\rightarrow h$ be a diagram isotopy. Then there is a sequence of diagram isotopies $H_1,\ldots,H_n$, such that $h_n\ldots h_1(D)$ coincides with $h(D)$ up to detour moves, and such that $H_i$ is the identity on a neighborhood of each real crossing but the $i$-th one.
\end{corollary}

\begin{remark}It is to be understood that the $i$-th crossing of $h_k\ldots h_1(D)$ is $h_k\ldots h_1(p_i)$.
\end{remark}

\begin{proof}
Any generalized braid is (boundary-fixing) homotopic to a braid $\beta\subset \Sigma\times\left[ 0,1\right]$ such that the $i$-th strand is vertical before the time $\frac{i-1}{n}$ and vertical again after the time $\frac{i}{n}$. Take such a braid $\beta$ that is homotopic to $\prescript{H\!\!}{}{\beta}$. Any diagram isotopy $H^\prime$ such that $\beta=\prescript{H^\prime\!\!}{}{\beta}$ factorizes into a product $H_n\ldots H_1$ satisfying the last required condition. The fact that $h_n\ldots h_1(D)$ and $h(D)$ coincide up to detour moves is a consequence of Proposition~\ref{prop:braid_homotopy}.
\end{proof}

\section{Virtual knot theory on a weighted group}\label{sec:VKTG}

In this section, we define a new Gauss diagram theory, that depends on an arbitrary group $\pi$ and a homomorphism $w:\pi\rightarrow\FF_2\simeq\ZZ/2\ZZ$. These two data together are called a \textit{weighted group}. When $(\pi,w)$ is the weighted fundamental group of a surface (Definition~\ref{def:wfg}), this theory encodes, fully and faithfully, virtual knot diagrams on that surface (Definition~\ref{def:virtualknotdiags}). 
\subsection{General settings and the main theorem}

\begin{definition}\label{def:GammaRm}
Let $\pi$ be an arbitrary group and $w$ a homomorphism from $\pi$ to $\FF_2.$ A \textit{Gauss diagram on $\pi$} is a classical Gauss diagram decorated with

\begin{itemize}
\item an element of $\pi$ on each edge if the diagram has at least one arrow.
\item a single element of $\pi$ \textit{up to conjugacy} if the diagram is empty.\end{itemize}

Such diagrams are subject to the usual types of R-moves, plus an additional \textit{conjugacy move}, or \textit{$w$-move} -- the dependence on $w$ arises only there. An equivalence class modulo all these moves is called a \textit{virtual knot type on $(\pi,w)$}.

A \textit{subdiagram} of a Gauss diagram on $\pi$ is the result of removing some of its arrows. Removing an arrow involves a merging of its ($2$, $3$, or $4$) adjacent edges, and each edge resulting from this merging should be marked with the product in $\pi$ of the former markings. If all the arrows have been removed, this product is not well-defined, but its conjugacy class is. 
\end{definition}

The notion of subdiagrams is useful to construct finite-type invariants (see Section~\ref{FTI}), but it already allows explicit understanding of
\begin{enumerate}
\item The distinction between empty and non empty diagrams in the definition above.
\item The \enquote{merge $\rightsquigarrow$ multiply} principle, which is omnipresent, in particular in R-moves.
\end{enumerate}

An \textbf{$\mathrm{R}_1$-move} is the local addition or removal of an isolated arrow, surrounding an edge marked with the unit $1\in\pi$. The markings of the affected edges must satisfy the rule indicated on Fig.\ref{pic:GammaRm} (top-left). There are no conditions on the decorations of the arrows.

\textit{Exceptional case:} If the isolated arrow is the only one in the diagram on the left, then the markings $a$ and $b$ on the picture actually correspond to the same edge, and the diagram on the right, with no arrow, must be decorated by $\left[ a\right] $, the conjugacy class of $a$.

\begin{figure}[h!]
\centering 
\psfig{file=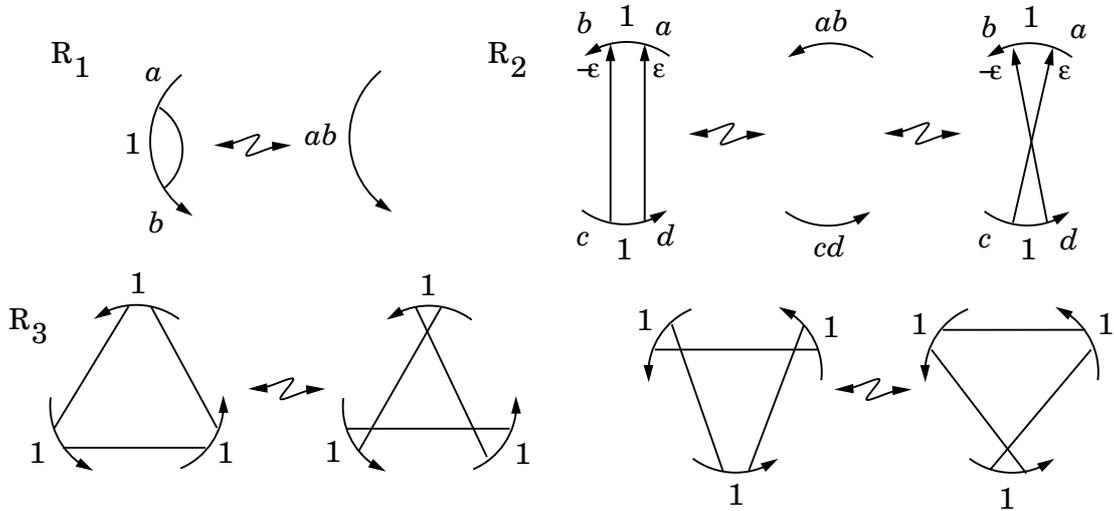
,scale=0.75}
\caption{The $\mathrm{R}$-moves for Gauss diagrams on a group -- the exceptional cases and the rules for the missing decorations are made precise in Definition~\ref{def:GammaRm}.}\label{pic:GammaRm}
\end{figure}

\begin{figure}[h!]
\centering 
\psfig{file=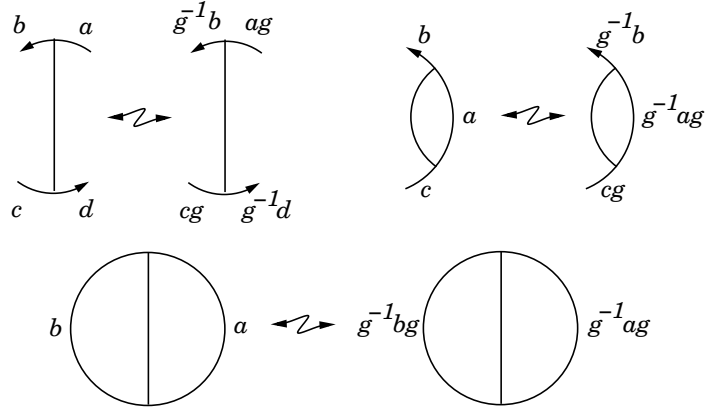
,scale=0.7}
\caption{The general conjugacy move (top-left) and its two exceptional cases -- in every case the orientation of the arrow switches if and only if $w(g)=-1$.}\label{pic:conjugacy}
\end{figure}

An \textbf{$\mathrm{R}_2$-move} is the addition or removal of two arrows with opposite writhes and matching orientations as shown on Fig.\ref{pic:GammaRm} (top-right). The surrounded edges must be decorated with $1$, and the \enquote{merge $\rightsquigarrow$ multiply} rule should be satisfied.\newline
\textit{Exceptional case of type 1:} If the markings $a$ and $d$ (\resp $b$ and $c$) correspond to the same edge, then the resulting marking shall be $cab$ (\resp $abd$).\newline
\textit{Exceptional case of type 2:} If the middle diagram contains no arrow at all, \ie $a$ and $d$ match and so do $b$ and $c$, then the (only) marking of the middle diagram shall be $\left[ ab\right] $.

An \textbf{$\mathrm{R}_3$-move} may be of the two types shown on Fig.\ref{pic:GammaRm} (bottom left and right). The surrounded edges must be decorated by $1$, the value of $\mathrm{w}(\cdot)\varepsilon(\cdot)$ must be the same for all three of them, and the values of $\uparrow\!\!(\cdot)$ must be pairwise distinct (see Definition~\ref{def_epsilon}).

A \textbf{conjugacy move} depends on an element $g\in\pi$. It changes the markings of the adjacent edges to an arbitrary arrow as indicated on Fig.\ref{pic:conjugacy}. Besides, if $w(g)=-1$ then the orientation of the arrow is reversed -- though its sign remains the same.

\begin{remark}
By composing R-moves and $w$-moves, it is possible to perform \textit{generalized moves}, which look like R-moves but depend on $w$. Fig.\ref{pic:generalizedGamma} shows some of them.
\end{remark}

\begin{figure}[h!]
\centering 
\psfig{file=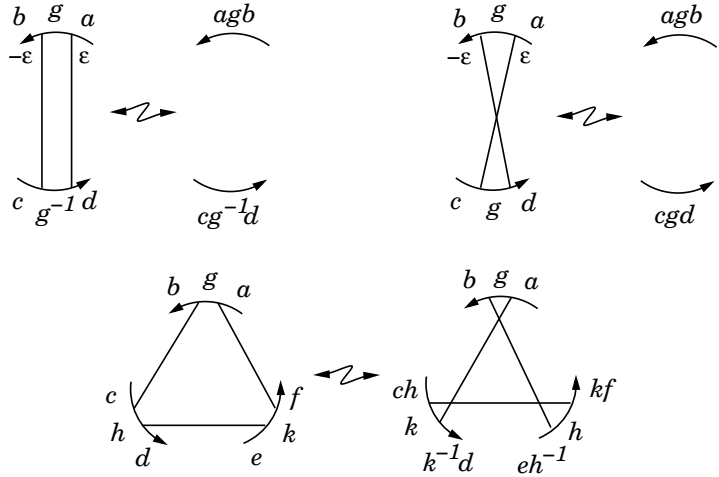
,scale=0.65}
\caption{Some generalized moves -- for the $\mathrm{R}_3$ picture, it is assumed that $ghk=1$. \textbf{Warning:} the rules for the arrow orientations here depend on the value of $w(g)$.}\label{pic:generalizedGamma}
\end{figure}

\begin{theorem}\label{thm:Th1}
Let $(\Sigma,x)$ be an arbitrary surface with a base point, and denote by $(\pi,w)$ the weighted fundamental group of $(\Sigma,x)$ (see Definition~\ref{def:wfg}). There is a $1-1$ correspondence $\Phi$ between Gauss diagrams on $\pi$ up to R-moves and $w$-moves (\ie virtual knot types on $(\pi,w)$), and virtual knot diagrams on $\Sigma$ up to diagram isotopy, Reidemeister moves and detour moves \mbox{(\ie virtual knot types on $\Sigma$).}
\end{theorem}

\begin{proof}
Fix a subset $X$ of $\Sigma$ homeomorphic to a closed $2$-dimensional disc and containing the base point $x$ -- so that $\pi=\pi_1(\Sigma,X)$. Also, $X$ being contractible allows one to fix a trivialization of the thickening of $\Sigma$ over $X$: this gives meaning to the locally \textit{over} and \textit{under} branches when a knot diagram has a real crossing in $X$.\newline

\textbf{Construction of the bijection.} Pick a knot diagram $D\in \Sigma$ and assume that  every real crossing of $D$ lies over $X$. Then $D$ defines a Gauss diagram on $\pi$, denoted by $\varphi (D)$: the signs of the arrows are given by the writhes, their orientation is defined by the trivialization of $M\rightarrow\Sigma$ over $X$, and each edge is decorated by the class in $\pi$ of the corresponding arc in $D$. This defines $\varphi (D)$ without ambiguity if $D$ has at least one real crossing. If it does not, then define $\varphi (D)$ as a Gauss diagram without arrows, decorated with the conjugacy class corresponding to the free homotopy class of $D$. Finally, put
$$\Phi (D):=\left[ \varphi (D)\right] \text{ $\operatorname{mod}$ R-moves and $w$-moves.}$$

\textbf{Invariance of $\Phi$ under diagram isotopy and detour moves.} It is clear from the definitions that $\varphi(D)$ is strictly unchanged under detour moves on $D$. Now assume that $D_1$ and $D_2$ are equivalent under \textit{usual} diagram isotopy -- that is, diagram isotopy that may take real crossings out of $X$ for some time. By Corollary~\ref{1by1}, it is enough to understand what happens for a diagram isotopy along which only one crossing goes out of $X$. In that case, $\varphi (D)$ is changed by a $w$-move performed on the arrow corresponding to that crossing, where the conjugating element $g$ is the loop followed by the crossing along the isotopy. Indeed, since the first Stiefel-Whitney class of the thickening of $\Sigma$ coincides with that of its tangent bundle, it follows that:
\begin{enumerate}
\item The orientation of the fibre (and thus the notions of \enquote{over} and \enquote{under}) is reversed along $g$ if and only if $w(g)=-1$, which actually corresponds to the rule for arrow orientations in a $w$-move.
\item The orientation of the fibre over the crossing is reversed along $g$ if and only if a given local orientation of $\Sigma$ is reversed along $g$, so that the writhe of the crossing never changes.
\end{enumerate}

\textbf{Invariance of $\Phi$ under Reidemeister moves.} Up to conjugacy by a diagram isotopy, it can always be assumed that a Reidemeister move happens inside $X$. In that case, at the level of $\varphi (D)$, it clearly corresponds to an R-move as described in Definition~\ref{def:GammaRm}.\newline

So far, $\Phi$ is a well-defined map from the set of virtual knot types on $\Sigma$ to the set of virtual knot types on $(\pi,w)$.

\textbf{Construction of an inverse map $\Psi$.} If $G$ is a Gauss diagram without arrows, then define $\psi (G)$ as the totally virtual knot with free homotopy class equal to the marking of $G$ -- it is well-defined up to detour moves. If $G$ has arrows, then for each of them draw a crossing inside $X$ with the required writhe, and then join these by totally virtual arcs with the required homotopy classes. The resulting diagram $\psi (G)$ is well-defined up to diagram isotopy and detour moves by this construction. In both cases, put
$$\Psi (D):=\text{virtual knot type of $\psi (D)$}.$$

Let us prove that $\varphi$ and $\psi$ are inverse maps, so that $\Psi$ will be the inverse of $\Phi$ as soon as it is invariant under R-moves and $w$-moves.

It is clear from the definitions that $\varphi\circ\psi$ coincides with the identity. It is also clear that $\psi\circ\varphi$ is the identity, up to detour moves, for \textit{totally virtual} knot diagrams. 

Now fix a knot diagram $D$ with at least one real crossing (and all real crossings inside $X$). Recall that $\psi\circ\varphi (D)$ is defined up to diagram isotopy and detour moves, so fix a diagram $D^\prime$ in that class. There is a natural correspondence between the set of real crossings of $D$ and those of $D^\prime$, due to the fact that both identify by construction with the set of arrows of 
$\varphi (D)$. Pick a diagram isotopy $h$ that takes each real crossing of $D$ to meet its match in $D^\prime$, \textit{without leaving $X$}. Then clearly $\varphi (h(D))=\varphi (D)$, and because $\varphi\circ\psi$ is the identity, one gets
\begin{equation}\label{proofmain}
\varphi(h(D))=\varphi(D^\prime).
\end{equation}
The choice of $h$ ensures that $h(D)$ and $D^\prime$ differ only by totally virtual arcs, and \eqref{proofmain} implies that each of these, in $h(D)$, has the same class in $\pi_1(\Sigma,X)$ as its match in $D^\prime$, which means by definition that $h(D)$ and $D^\prime$ are equivalent up to detour moves. Thus $\psi\circ\varphi$ is the identity up to diagram isotopy and detour moves.\newline

\textbf{Invariance of $\Psi$ under R-moves.}
Let us treat only the case of $\mathrm{R}_2$-moves, which contains all the ideas. Let $G_1$ and $G_2$ differ by an $\mathrm{R}_2$-move, and assume that $G_1$ is the one with more arrows. By appropriate diagram isotopy and detour moves \textit{inside} $X$, performed on $\psi(G_1)$, it is possible to make the two concerned crossings \enquote{face} each other, as in Fig.\ref{pic:R2} (left). The paths $\alpha_1$ and $\alpha_2$  from this picture are totally virtual and  trivial in $\pi_1 (\Sigma,X)$, thus $\psi(G_1)$ is equivalent to the second diagram of Fig.\ref{pic:R2} up to detour moves. The fact that at this point, an R-II move is actually possible is a consequence of (in fact equivalent to) the combinatorial conditions defining the R-moves. Denote by $D$ the third diagram of the picture. The \enquote{merge $\rightsquigarrow$ multiply} principle that rules $\mathrm{R}_2$-moves implies that $\varphi(D)=G_2$, so that
\begin{equation}\label{proofmain2}
\psi(G_1)\sim D\sim\psi\circ
\varphi(D)=\psi(G_2),\end{equation}
where $\sim$ is the equivalence under diagram isotopy, detour moves and Reidemeister moves. It follows that $\psi(G_1)$ and $\psi(G_2)$ have the same knot type.
 
\begin{figure}[h!]
\centering 
\psfig{file=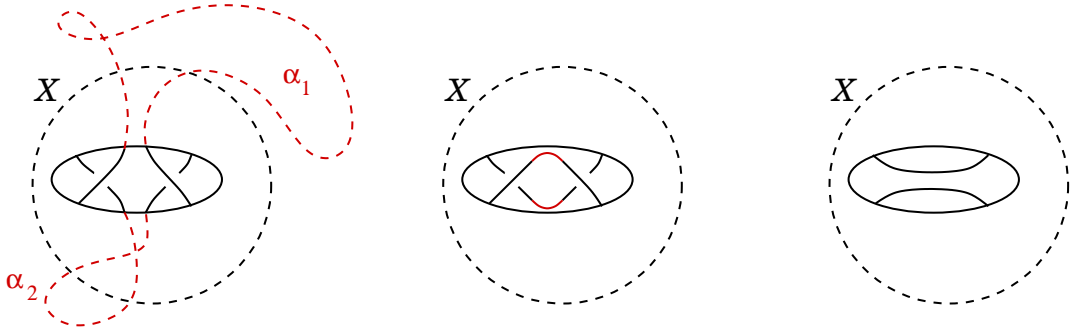}
\caption{$\mathrm{R}_2$-moves actually correspond to Reidemeister moves}\label{pic:R2}
\end{figure}

\begin{figure}[h!]
\centering 
\psfig{file=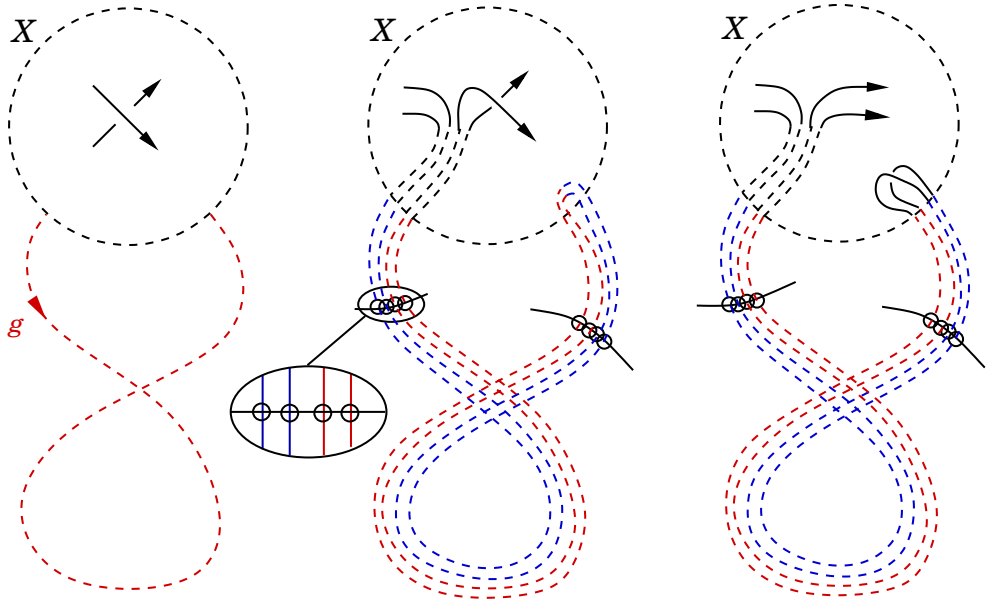}
\caption{Performing a $w$-move -- the railway trick}\label{pic:wmove}
\end{figure}

\textbf{Invariance of $\Psi$ under $w$-moves.} Let $G_1$ and $G_2$ differ by a $w$-move on $g\in\pi$. Call $c$ the corresponding crossing on the diagram $\psi(G_1)$. Then, pick two little arcs right before $c$, one on each branch, and make them follow $g$ by a detour move. At the end, one shall see a totally virtual $4$-lane railway as pictured on Fig.\ref{pic:wmove} (middle): the strands are made parallel, \ie any (virtual) crossing met by either of them is part of a larger picture as indicated by the zoom. This ensures that, using the mixed version of Reidemeister III moves, one can slide the real crossing all along the red part of the railway, ending with the diagram on the right of the picture -- let us call it $D$. The conclusion is identical to that for R-moves: again $\varphi(D)=G_2$ and  \eqref{proofmain2} holds, whence $\psi(G_1)$ and $\psi(G_2)$ have the same knot type.

\end{proof}

\subsubsection{About the orbits of $w$-moves}\label{orbits}

It could feel natural to try to get rid of $w$-moves by understanding their orbits in a synthetic combinatorial way. This is what is done in Section~\ref{subsec:abelian} in the particular case of an abelian group $\pi$ endowed with the trivial homomorphism $\pi\rightarrow\FF_2$.

In general, for a Gauss diagram on $\pi$, $G$, denote by $h_1(G)$ the \textit{set} of free homotopy classes of loops in the underlying topological space of $G$ (it is the set of conjugacy classes in a free group on $deg(G)+1$ generators). Also, denote by $h_1(\pi)$ the set of conjugacy classes in $\pi$. Then the $\pi$-markings of $G$ define a map
$$F_G:h_1(G)\rightarrow h_1(\pi).$$
Observe that the map $G\mapsto F_G$ is invariant under $w$-moves. This raises a number of questions that amout to technical group theoretic problems, and which will not be answered here ($G^w$ denotes the orbit of $G$ under $w$-moves):

\begin{enumerate}
\item Is the map $G^w\mapsto F_G$ injective?
\item If the answer to $1.$ is yes, then is $G^w$ determined by a finite number of values of $F_G$, for instance its values on the free homotopy classes of \textit{simple loops}?
\item Is it possible to detect in a simple manner what maps $h_1(G)\rightarrow h_1(\pi)$ lie in the image of $G^w\mapsto F_G$?
\end{enumerate}

\begin{remark}
Gauss diagrams with decorations in $h_1(\Sigma)$ can be met for example in \cite{Grishanov}, where they are used to construct knot invariants in a thickened oriented surface $\Sigma$ -- see also Section~\ref{sec:Grishanov}. If the answer to Question $1.$ above is no, then such invariants, which factor through $F_G$, stand no chance to be complete.
\end{remark}

\begin{remark}Even for diagrams with only one arrow, it still does not seem easy to answer the \enquote{simple loop} version of Question $2.$ Given $x,y,h,k$ in a finite type free group, is it true that $$\begin{array}{cccc}
hxh^{-1}kyk^{-1}=xy
& \Longrightarrow & \exists l, \,  \left\lbrace \begin{array}{ccc}
hxh^{-1} & = & lxl^{-1} \\ kyk^{-1} & = & lyl^{-1}
\end{array}\right. & ?\end{array}$$
\end{remark}

Let us end with an example that shows that the values of $F_G$ on the (finite) set of simple loops running along at most one arrow is not enough (cf. Question $2.$). Fig.\ref{exWP1} shows a Gauss diagram with such decorations -- $\left\lbrace a,b \right\rbrace$ is a set of generators for the free group $\pi_1(\Sigma)\simeq \FF (a,b)$, where $\Sigma$ is a $2$-punctured disc.  These particular values of $F_G$ do not determine the free homotopy class of the red loop $\gamma$, as it is shown in Fig.\ref{exWP2}.

\begin{figure}[h!]
\centering 
\psfig{file=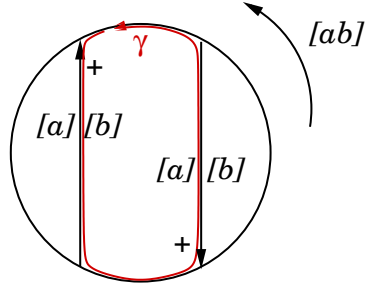,scale=0.78}
\caption{A Gauss diagram with $h_1$-decorations that does not define a unique virtual knot}\label{exWP1}
\end{figure}

In fact, these two virtual knots are even distinguished by Vassiliev-Grishanov's planar chain invariants, which means they represent different virtual knot types.

\begin{figure}[h!]
\centering 
\psfig{file=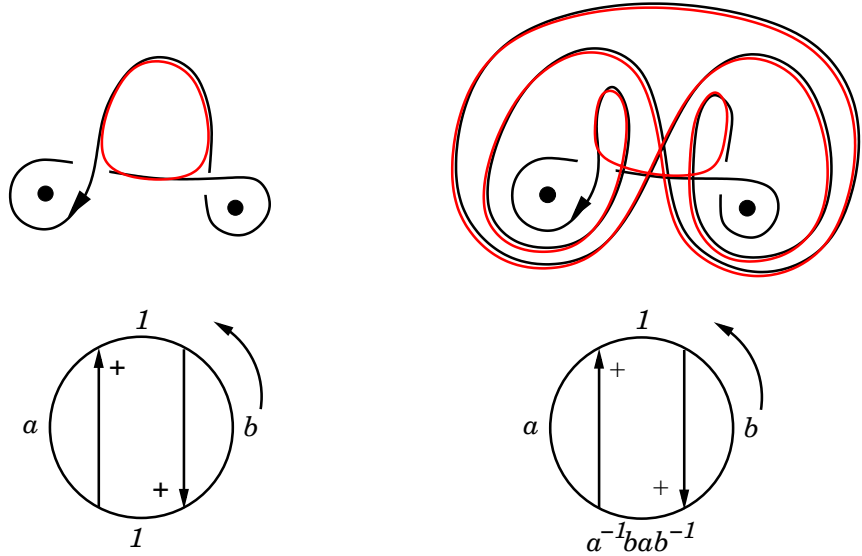,scale=1}
\caption{One red loop is trivial, while the other is a commutator}\label{exWP2}
\end{figure}

\subsection{Abelian Gauss diagrams}
\label{subsec:abelian}

In this subsection, $\pi$ is assumed to be abelian, and $w_0$ denotes the trivial homomorphism $\pi\rightarrow\FF_2$. We describe a version of Gauss diagrams that carries as much information as the previously introduced virtual knot types on $(\pi,w_0)$, with two improvements:
\begin{itemize}
\item The diagrams are made of less data than in the general version.
\item This version is free from conjugacy moves.
\end{itemize}
It is inspired from the decorated diagrams introduced by T. Fiedler to study combinatorial invariants for knots in thickened surfaces (see \cite{Fiedler,Fiedlerbraids} and also \cite{MortierPolyakEquations}). 

We use the same notation $G$ for a Gauss diagram and its underlying topological space, which has a $1$-dimensional complex structure with edges and arrows as oriented $1$-cells. $H_1(G)$ denotes its first integral homology group.

\begin{deflemma}[fundamental loops]\label{def:distinguished}
Let $G$ be a classical Gauss diagram of degree $n$. There are exactly $n+1$ simple loops in $G$ respecting the local orientations of edges and arrows, and going along at most one arrow. They are called the \textit{fundamental loops} of $G$ and their homology classes form a basis of $H_1(G)$.
\end{deflemma}

\begin{definition}[abelian Gauss diagram]\label{def:ab}
Let $\pi$ be an abelian group. An \textit{abelian Gauss diagram on $\pi$} is a classical Gauss diagram $G$ decorated with a group homomorphism $\mu:H_1(G)\rightarrow\pi$. It is usually represented by its values on the basis of fundamental loops, that is, one decoration in $\pi$ for each arrow, and one for the base circle -- that last one is called the \textit{global marking of $G$}.

A Gauss diagram on $\pi$ determines an abelian Gauss diagram as follows:
\begin{itemize}
\item The underlying classical Gauss diagram is the same.
\item Each fundamental loop is decorated by the sum of the markings of the edges that it meets (see Fig~\ref{pic:ab}).
\end{itemize}

\begin{figure}[h!]
\centering 
\psfig{file=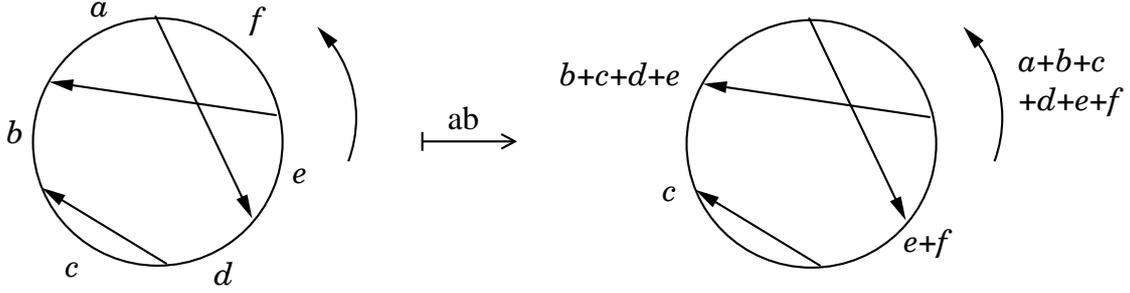,scale=0.85}
\caption{Abelianizing a Gauss diagram on an abelian group}\label{pic:ab}
\end{figure}

This defines an \textit{abelianization map} $\ab$.

\end{definition}

\begin{proposition}\label{prop:abelian}
The map $\ab$ induces a natural $1-1$ correspondence between abelian Gauss diagrams on $\pi$ and equivalence classes of Gauss diagrams on $\pi$ up to $w_0$-moves. Moreover, if $\pi=\pi_1(\Sigma)$ is the fundamental group of a surface, then these sets are in $1-1$ correspondence with the set of virtual knot diagrams on $\Sigma$ up to diagram isotopy and detour moves. 
\end{proposition}

\begin{proof}

The proof of the last statement is contained in that of Theorem~\ref{thm:Th1} -- through the facts that $\phi$ and $\psi$ are inverse maps up to detour moves and diagram isotopy, and that $w$-moves at the level of knot diagrams can be performed using only detour moves and diagram isotopies, by the railway trick (Fig.\ref{pic:wmove}).

As for the first statement, one easily sees that $\ab$ is invariant under $w_0$-moves. We have to show that conversely, if $\ab(G_1)=\ab(G_2)$, then $G_1$ and $G_2$ are equivalent under $w_0$-moves.

This is clear if $G_1$ has no arrows, since then $\ab(G_1)=G_1$. Now proceed by induction. Since $G_1$ and $G_2$ have the same abelianization, they have in particular the same underlying classical Gauss diagram, and there is a natural correspondence between their arrows.

\textbf{Case 1:} No two arrows in $G_1$ cross each other. Then at least one arrow surrounds a single isolated edge on one side (as in an $\mathrm{R}_1$-move). Choose such an arrow $\alpha$ and remove it, as well as its match in $G_2$. By induction, there is a sequence of $w_0$-moves on the resulting diagram $G_1^\prime$ that turns it into $G_2^\prime$.
Since the arrows of $G_1^\prime$ have a natural match in $G_1$, those $w_0$-moves make sense there, and take every marking of $G_1$ to be equal to its match in $G_2$, except for those in the neighborhood of $\alpha$. So we may assume that $G_1$ and $G_2$ only differ near $\alpha$ as in Fig.\ref{pic:proofab}. Since all the unseen markings coincide in $G_1$ and $G_2$, and since $\ab(G_1)$ and $\ab(G_2)$ have the same global marking, it follows that $$a+b+c=a^\prime+
b^\prime+c^\prime.$$ Thus a $w_0$-move on $\alpha$ with conjugating element $g=a^\prime-a$ turns $G_1$ into $G_2$.

\begin{figure}[h!]
\centering 
\psfig{file=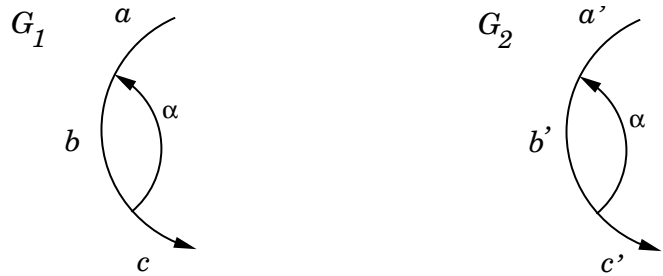
,scale=0.8}
\caption{Notations for case $1$}\label{pic:proofab}
\end{figure}

\textbf{Case 2:} There is at least one arrow $\alpha$ in $G_1$ that intersects another arrow. By the same process as in case $1$, one may assume that $G_1$ and $G_2$ only differ near $\alpha$ -- see Fig.\ref{pic:proofabis}, where $a$, $b$, $c$ and $d$ actually correspond to pairwise distinct edges since $\alpha$ intersects an arrow. Again, since all the unseen markings coincide in $G_1$ and $G_2$, one obtains
$$a+d=a^\prime+d^\prime,$$ and $$b+c=b^\prime+c^\prime,$$ by considering the global marking, and the marking of $\alpha$, in $\ab(G_1)$ and $\ab(G_2)$. Moreover, there is at least one arrow intersecting $\alpha$: considering the marking of that arrow gives $$a+b=a^\prime+b^\prime.$$ The last three equations may be written as $$a^\prime-a=b-b^\prime=c^\prime-c=d-d^\prime,$$ so that, again, a $w_0$-move on $\alpha$ with conjugating element $g=a^\prime-a$ turns $G_1$ into $G_2$.

\begin{figure}[h!]
\centering 
\psfig{file=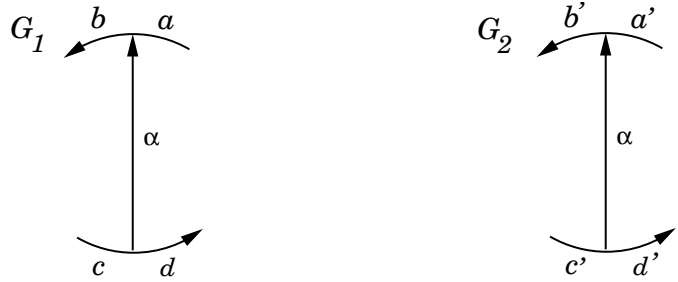
,scale=0.8}
\caption{Notations for case $2$}\label{pic:proofabis}
\end{figure}

\end{proof}

\begin{remark}
A different proof of this proposition was given in a draft paper, in the special case $\pi=\ZZ$ (\cite{MortierGaussDiagrams}, Proposition $2.2$). As an exercise, one can show that this proof extends to the case of an arbitrary abelian group.
\end{remark}

To make the picture complete, it only remains to understand R-moves in this context.

\begin{definition}[obstruction loops]
Within any local Reidemeister picture like those shown on Fig.\ref{Rmoves} featuring at least one arrow, there is exactly one (unoriented) simple loop. We call it the \textit{obstruction loop}. 
Fig.\ref{pic:loops} shows typical examples.
\end{definition}

\begin{definition}[R-moves]\label{lem:obstruction}
A move from Fig.\ref{Rmoves} is likely to define an R-move only if the obstruction loop lies in the kernel of the decorating map $H_1(G)\rightarrow \pi$ (which makes sense even though the loop is unoriented). Under that assumption, the \emph{R-moves for abelian Gauss diagrams} are defined by the usual conditions:

\begin{itemize}
\item $i=1.$ No additional condition.
\item $i=2.$ The arrows head to the same edge, and have opposite signs.

\item $i=3.$ The value of $\mathrm{w}(e)\varepsilon(e)$ is the same for all three visible edges $e$, and the values of  $\uparrow\!\!(e)$ are pairwise different (see Definition~\ref{def_epsilon}).
\end{itemize}

\end{definition}

\begin{figure}[h!]
\centering 
\psfig{file=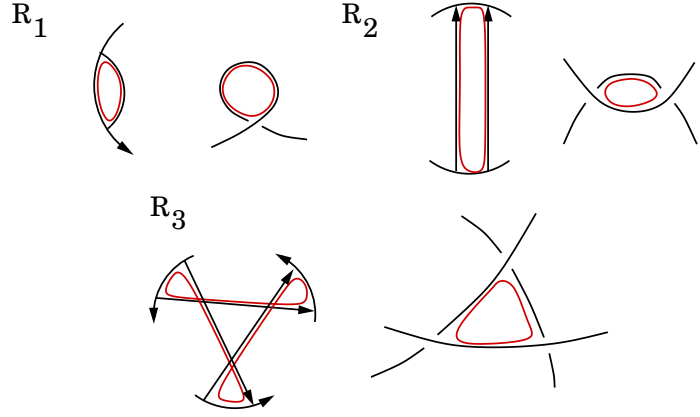,scale=0.70}
\caption{Homological obstruction to $\mathrm{R}$-moves}\label{pic:loops}
\end{figure}

\begin{theorem}\label{thm:abelian}
The map $\ab$ induces a natural $1-1$ correspondence between equivalence classes of abelian Gauss diagrams on $\pi$ up to R-moves and virtual knot types on $(\pi,w_0)$.
\end{theorem}

\begin{proof}
$\ab$ clearly maps an R-move in the non commutative sense to an R-move in the abelian sense. Conversely, if $\ab(G_1)$ and $\ab(G_2)$ differ from an (abelian) R-move, then the vanishing homological obstruction implies that $G_1$ and $G_2$  are in a position to perform a \enquote{generalized R-move} like the examples pictured on Fig.\ref{pic:generalizedGamma}.
\end{proof}

Theorems~\ref{thm:Th1}~and ~\ref{thm:abelian} together imply the following

\begin{corollary}\label{cor:orientable}
If $\Sigma$ is an orientable surface with abelian fundamental group, then there is a $1-1$ correspondence between abelian Gauss diagrams on $\pi_1(\Sigma)$ up to R-moves, and virtual knot types on $\Sigma$.
\end{corollary}

\subsection{Homological formulas}
\label{subsec:homform}

It may seem not easy to compute an arbitrary value of the linear map decorating an abelian Gauss diagram, given only its values on the fundamental loops. To end this section, we give two formulas to fill this gap, by understanding the coordinates of an arbitrary loop in the basis of fundamental loops.

\subsubsection{The energy formula}

Fix an abelian Gauss diagram $G$. Observe that as a cellular complex, $G$ has no $2$-cells, thus every $1$-homology class has a unique set of \enquote{coordinates} along the family of edges and arrows. For each $1$-cell $c$ (which may be an arrow or an edge), we denote by $\left<\cdot,c\right>:H_1(G)\rightarrow\ZZ$ the coordinate function along $c$. It is a group homomorphism.

Let us denote by $\left[ A\right] \in H_1(G)$ the class of the fundamental loop associated with an arrow $A$ (Fig.\ref{pic:energy} left).

\begin{deflemma}[Energy of a loop]\label{def:energy}
Fix an edge $e$ in $G$, and a class $\gamma\in H_1(G)$. The value of 
\begin{equation}\label{eq:energy} 
E_e (\gamma) = \left<\gamma,e\right> -\sum_{\left<\left[ A\right],e\right>=1} \left<\gamma,A\right>
\end{equation}
is independent of $e$. This defines a group homomorphism  $E:H_1(G) \rightarrow \ZZ$.
\end{deflemma}

\begin{proof}
Let us compare the values of $E_\cdot (\gamma)$ for an edge $e$ and the edge $e^\prime$ right after it. $e$ and $e^\prime$ are separated by a vertex $P$, which is the endpoint of an arrow $A$. There are two possible situations (Fig.\ref{pic:energy}):
\begin{enumerate}
\item $P$ is the tail of $A$. Then $\left<\left[ A\right],e\right>=1$ and $\left<\left[ A\right],e^\prime\right>=0$, so that
$$E_e (\gamma)-E_{e^\prime} (\gamma)=\left<\gamma,e\right>
-\left<\gamma,A\right>
-\left<\gamma,e^\prime\right>.$$

\item $P$ is the head of $A$. Then $\left<\left[ A\right],e\right>=0$ and $\left<\left[ A\right],e^\prime\right>=1$, so that
$$E_e (\gamma)-E_{e^\prime} (\gamma)=\left<\gamma,e\right>
+\left<\gamma,A\right>
-\left<\gamma,e^\prime\right>.$$
\end{enumerate}
In both cases, $E_e (\gamma)-E_{e^\prime}(\gamma)$ is equal to $\left<\partial
\gamma,P\right>$, which is $0$ since $\gamma$ is a cycle.
\end{proof}

\begin{figure}[h!]
\centering 
\psfig{file=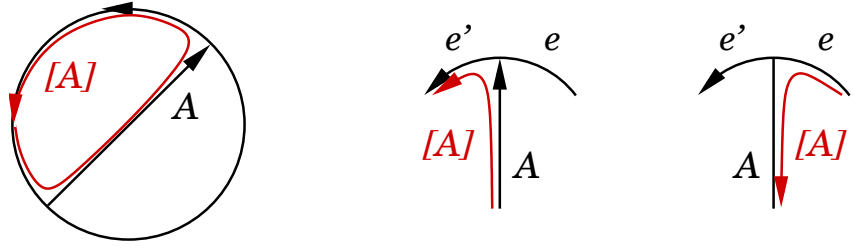,
scale=1}
\caption{The fundamental loop of an arrow and the two cases in the proof of Lemma~\ref{def:energy}}\label{pic:energy}
\end{figure}

\begin{theorem}\label{thm:formula2}
For any $\gamma\in H_1(G)$, one has the decomposition

\begin{equation}\label{eq:formula2} \gamma =\sum_A\left<\gamma,A\right> \left[ A\right] +E(\gamma)  \left[K \right] .\end{equation}
\end{theorem}

\begin{proof}
This formula is an identity between two group homomorphisms, so it suffices to check it on the basis of fundamental loops, which is immediate.
\end{proof}

\begin{remark}
The existence of a map $E$ such that Theorem~\ref{thm:formula2} holds was clear, since for each arrow $A$ considered as a $1$-cell, $[A]$ is the only fundamental loop that involves $A$. With that in mind, one may read into \eqref{eq:energy} as follows: $E(\gamma)$ counts the (algebraic) number of times that $\gamma$ goes through an edge, minus the number of those times that are already taken care of by the fundamental loops of the arrows. This number has to be the same for all edges, so that one recovers a multiple of $\left[K \right]$. \end{remark}

\subsubsection{The torsion formula}

Looking at \eqref{eq:formula2} and Fig.\ref{pic:energy}, one may feel that it would be more natural to have $\left[K \right]-\left[A \right]$ involved in the formula, instead of $\left[A \right]$, for all arrows $A$ such that $\left<\gamma,A\right>$ is negative -- that is, when $\gamma$ runs along $A$ with the wrong orientation more often than not. The formula then becomes
\begin{equation}\label{eq:torsion}
 \gamma=\sum_{\left<\gamma,A\right> >0}\left<\gamma,A\right> \left[ A\right] 
+ \sum_{\left<\gamma,A\right> <0}\left<\gamma,A\right> \left(\left[ A\right]- \left[ K\right]\right)\,-\mathcal{T}(\gamma)\left[ K\right],
\end{equation}
where 

\begin{equation}\label{TvsE}
-\mathcal{T}(\gamma)= E(\gamma)+ \sum_{\left<\gamma,A\right> <0}\left<\gamma,A\right>.
\end{equation}

\begin{definition}
$\mathcal{T}(\gamma)$ is called the \textit{torsion} of $\gamma$.
\end{definition}

How is \eqref{eq:torsion} different from \eqref{eq:formula2}?\newline
$\ominus$ On the negative side, unlike the energy, $\mathcal{T}$ is not a group homomorphism. But it actually behaves almost like one:

\begin{lemma}\label{lem:torsionmorphism}
Let $\gamma_1$ and $ \gamma_2$ be two homology classes such that
$$\forall A,\,\, \left<\gamma_1,A\right>
\left<\gamma_2,A\right> \geq 0.$$
Then
$$\mathcal{T}(\gamma_1+\gamma_2)=\mathcal{T}(\gamma_1)+\mathcal{T}(\gamma_2).$$
\end{lemma}
\begin{proof}
It follows from the definition and the fact that $E(\gamma)$ is a homomorphism.
\end{proof}

\noindent $\oplus$ On the positive side:
\begin{lemma}\label{lem:torsionindep}
The torsion of a loop in a Gauss diagram $G$ does not depend on the orientations of the arrows of $G$.
\end{lemma}
\begin{proof}
By expanding the defining formula,
$$\mathcal{T}(\gamma)=
-\left<\gamma,e\right>
\hspace*{0.7cm}+ \sum_{\begin{array}{c}
\left<\gamma,A\right> <0\\
 \left<[A],e\right> =0
\end{array}}\left<\gamma,A\right> \hspace*{0.7cm}- \sum_{\begin{array}{c}
\left<\gamma,A\right> >0\\
 \left<[A],e\right> =1
\end{array}}
 \left<\gamma,A\right>,$$
one sees that reversing an arrow makes its contribution (if non zero) switch from one sum to the other, while $\left<\gamma,A\right>$ also changes signs.
\end{proof}

This lemma allows one to expect that $\mathcal{T}(\gamma)$ should admit a very simple combinatorial interpretation. It actually does, but only for a certain family of loops -- the ERS loops defined below. Fortunately enough, this family happens to positively generate $H_1(G)$, which allows one to compute the torsion of any loop by using Lemma~\ref{lem:torsionmorphism}.

\begin{definition}\label{def:ER}
The notation $\gamma$ is used for loops as well as $1$-homology classes.
A homology class $\gamma\in H_1(G)$ is said to be
\begin{itemize}

\item \textit{ER} (for \enquote{edge-respecting}), if for every edge $e$, $\left<\gamma,e\right> \geq 0.$

\item \textit{simple} if it is the class of a simple (injective) loop, that is, $\abs{\left<\gamma,c\right>} \leq 1$ for every $1$-cell $c$ (edge or arrow).

\item \textit{ERS} if it is ER and simple.

\item \textit{proper} if it runs along at least one arrow.

\end{itemize}
\end{definition}

\begin{figure}[h!]
\centering 
\psfig{file=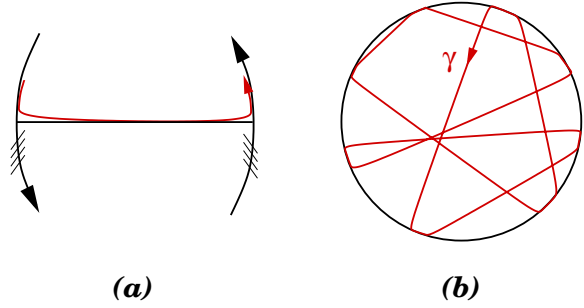,scale=0.72}
\caption{The local and global look of a proper ERS loop}\label{simpleloop}
\end{figure}

Consider a permutation $\sigma\in\mathfrak{S}\left( \llbracket 1,n\rrbracket\right)$, and set
$$\,\nearrow \!(\sigma):=\sharp\left\lbrace i\in\llbracket 1,n\rrbracket\mid\sigma(i)>i\right\rbrace.$$
It is easy to check that if $\sigma_0$ is the circular permutation $(1\,2\,\ldots\,n)$, then
$$\forall \sigma\in\mathfrak{S}, \,\nearrow \!(\sigma)=\,\nearrow \!(\sigma_0\sigma\sigma_0^{-1}).$$

\begin{definition}
The invariance property from above means that $\mathcal{T}$ is well-defined for permutations of a set of $n$ points lying in an abstract oriented circle.
We still denote this function by $\mathcal{T}$, and call it the \textit{torsion} of a permutation.
\end{definition}

Let $\gamma$ be a proper simple loop, then the set of edges $e$ such that $\left<\gamma,e\right> \neq 0$ can be naturally assimiliated to a finite subset of an oriented circle, and $\gamma$ induces a permutation of this set. Let us denote it by $\sigma_\gamma$.

\begin{theorem}\label{thm:torsion}
For all proper ERS loops $\gamma$,
$$\mathcal{T}(\gamma)=\,\nearrow \!(\sigma_\gamma).$$
\end{theorem}

This theorem can be useful in practice, since the torsion of a permutation can be computed at a glance on the braid-like presentation. Observe that
\begin{enumerate}
\item Every non proper loop is homologous to a multiple of $[K]$, easy to determine.
\item For every proper loop $\gamma$, there is an integer $n$ such that $\tilde\gamma=\gamma+n[K]$ is proper, ER, and has zero coordinate along at least one edge. Namely, $n=-\operatorname{min}_e
\left<\gamma,e\right>.$
\item Every class $\tilde\gamma$ as above may be decomposed as a sum $\tilde\gamma=\sum_i\gamma_i$ such that 
\begin{itemize}
\item all the $\gamma_i$'s are proper and ERS
\item $\forall i,j, A, \left<\gamma_i,A\right>
\left<\gamma_j,A\right> \geq 0$
\end{itemize}
\item By Lemma~\ref{lem:torsionmorphism}, $\mathcal{T}(\tilde{\gamma})=\sum_i
\mathcal{T}(\gamma_i)$, and the $\mathcal{T}(\gamma_i)$'s are given by Theorem~\ref{thm:torsion}.
\end{enumerate}

This shows that it is possible to compute any homology class by using the torsion formula. Whether it is more interesting than the energy formula depends on the context.

\begin{proof}[Proof of Theorem~\ref{thm:torsion}]
One may assume that for every arrow $A$, $\left< \gamma,A \right>=1.$
Indeed, deleting an arrow avoided by $\gamma$, or reversing the orientation of an arrow that $\gamma$ runs in the wrong direction, have no effect on either side of the formula (notably because of Lemma~\ref{lem:torsionindep}). Under this assumption, half of the edges of $G$ are run by $\gamma$: call them the \textit{red edges of $G$}, while the other half are called the \textit{blue edges}. Red and blue edges alternate along the orientation of the circle.

If $e$ is any (red or blue) edge, we define:
$$\lambda(e):=\sum_A \left<[A],e\right>. $$

\begin{lemma}\label{l1}
Under the assumption that $\left< \gamma,A \right>=1$ for all $A$, the value of $\lambda(e)$ only depends on the color of the edge $e$. Moreover, $$\lambda(blue)=\lambda(red)-1=\,\nearrow \!(\sigma_\gamma).$$
\end{lemma}

Let us temporarily admit this result. By the definition of $\lambda$,
$$\begin{array}{rcl}
\sum_A\left[ A\right] & = & \sum \text{arrows} +\lambda(\text{red}) \sum ( \text{red edges}) 
+\lambda(\text{blue}) \sum ( \text{blue edges})\\

& \stackrel{\text{Lemma}~
\ref{l1}}{=} & \sum \text{arrows} + \sum \left(\text{red edges}  \right) + \lambda(\text{blue}) \sum \left(\text{red and blue edges} \right)\\
& \stackrel{\phantom{\text{Lemma}~
\ref{l1}}}{=} & \gamma + \lambda(\text{blue}) [K]\\
& \stackrel{\text{Lemma}~
\ref{l1}}{=} & \gamma +\,\nearrow \!(\sigma_\gamma) [K].
\end{array}$$
Since it was assumed that $\left< \gamma,A \right>=1$ for every arrow, the definition of $\mathcal{T}$ \eqref{eq:torsion} reads
$$\gamma=\sum_A\left[ A\right]-\mathcal{T}(\gamma) [K],$$ which terminates the proof of the theorem, up to Lemma~\ref{l1}.

\end{proof}

\begin{proof}[Proof of Lemma \ref{l1}]
In the case of $\sigma_0=(1\,2\,\ldots\,n)$ depicted on Fig.\ref{h1}, it is easy to see that $\lambda(red)=n$ and $\lambda(blue)=n-1$, while $\,\nearrow \!(\sigma_0)=n-1$.
The lemma being true for one diagram, let us show that it survives elementary changes that cover all the diagrams.

\begin{figure}[h!]
\centering 
\psfig{file=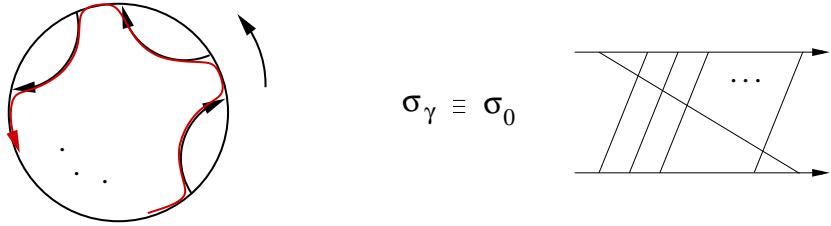,scale=0.8}
\caption{Braid-like representations of permutations are to be read from bottom to top}\label{h1}
\end{figure}

Notice that for every proper ERS loop $\gamma$, $\sigma_\gamma$ is a \textit{cycle}, and conversely a permutation that is a cycle uniquely defines an undecorated Gauss diagram \textit{and} a proper ERS loop $\gamma$ such that for every arrow $A$, $\left< \gamma,A \right>=1$. Thus, covering all possible permutations implies covering all possible diagrams and proper ERS loops. So all we have to check is that the formula survives an operation on $\sigma_\gamma$, of the form: 
$$\left(\, \ldots \, i\, j \, \ldots\,\right) \longrightarrow 
\left(\, \ldots \, j\, i \, \ldots\,\right)$$
The corresponding move at the level of Gauss diagrams may be of six different types, grouped in three pairs of reverse operations (Fig.\ref{h2}).

\begin{figure}[h!]
\centering 
\psfig{file=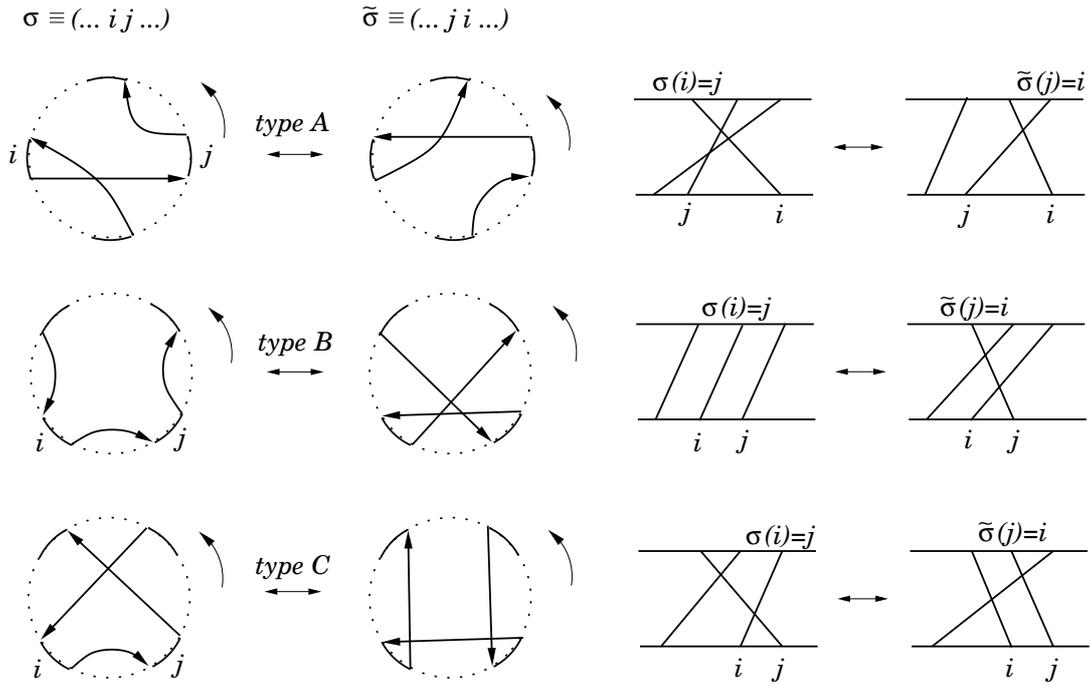,scale=1}
\caption{Twist moves on Gauss diagrams}\label{h2}
\end{figure}

On each diagram in Fig.\ref{h2}, the three moving arrows split the base circle into six regions. One computes the variation of $\lambda$ separately for each of these regions, and sees that it is the same for each of them. The results are gathered in the following table, proving the lemma.

$$
\begin{array}{c|c|c}
\text{type of move}     & \text{variation of }\lambda    & \text{variation of }\mathcal{T}(\gamma)\\ 
\hline A                               & \text{unchanged}       & \text{unchanged}      \\
B \text{ (from left to right)}  & \text{decreases by }1  & \text{decreases by }1  \\
C \text{ (from left to right)}  & \text{decreases by }1  & \text{decreases by }1 \\ \end{array}
$$

\end{proof}

\section{Finite-type invariants}\label{FTI}

One of the main points of using Gauss diagrams is their ability to describe finite-type invariants by simple formulas \citep{PolyakViro, Fiedler,  ChmutovKhouryRossi, PolyakChmutovHOMFLYPT}. In the case of classical long knots in $3$-space, such formulas actually cover all Vassiliev invariants as was shown by M.Goussarov \citep{Goussarov}. In the virtual case, the two notions actually differ (see \citep{KauffmanVKT99} and also \citep{ChrismanGPV, ChrismanLattices}. Finite type invariants for virtual knots that do admit Gauss diagram formulas shall be called GPV invariants \citep{GPV}. 

In \citep{MortierPolyakEquations}, a simple set of criteria was given to detect a particular family of those formulas, called \textit{virtual arrow diagram formulas}. Most of the examples that are known belong to this family. That includes Chmutov-Khoury-Rossi's formulas for the coefficients of the Conway polynomial \citep{ChmutovKhouryRossi} (and their generalization by M. Brandenbursky \citep{Brandenbursky}), as well as the formulas from \citep{Fiedler, Fiedlerbraids, Grishanov} where different kinds of decorated diagrams are used. Note however that the formulas for the invariants extracted from the HOMFLYPT polynomial \citep{PolyakChmutovHOMFLYPT} are arrow diagram formulas only if the variable $a$ is specialized to $1$ (which yields back the result of \citep{ChmutovKhouryRossi}).

In this section, we extend the results from \citep{MortierPolyakEquations} to an arbitrary surface. Then we show how to apply them to any other kind of decorated diagrams found in the literature, by defining \textit{symmetry-preserving maps} which enable one to jump from one theory to another.

\subsection{General algebraic settings}\label{sec:GDspaces}
We denote by $\mathfrak{G}_n$ (\resp $\mathfrak{G}_{\leq n}$) the $\QQ$-vector space freely generated by Gauss diagrams on $\pi$ of degree $n$ (\resp ${\leq n}$), and set $\mathfrak{G}= \varinjlim\mathfrak{G}_{\leq n}$. Unless $\pi$ is a finite group, these spaces are not finitely generated, and we define their hat versions $\widehat{\mathfrak{G}}_n$ (\resp $\widehat{\mathfrak{G}}_{\leq n}$) as the $\QQ$-spaces of formal series of Gauss diagrams of degree $n$ (\resp $\leq n$). Finally, set $\widehat{\mathfrak{G}}= \varinjlim\widehat{\mathfrak{G}}_{\leq n}.$ An arbitrary element of $\widehat{\mathfrak{G}}$ is usually denoted by $\mathcal{G}$ and called a \textit{Gauss series}, of degree $n$ if it is represented in $\widehat{\mathfrak{G}}_{\leq n}$ but not in $\widehat{\mathfrak{G}}_{\leq n-1}$. The notation $G$ is saved for single Gauss diagrams. 

A Gauss diagram $G$ of degree $n$ has a \textit{group of symmetries} $\operatorname{Aut}(G)$, which is a subgroup of $\ZZ/2n$, made of the rotations of the circle that leave unchanged a given representative of $G$ (see Subsection~\ref{sec:Sinj}).
$\mathfrak{G}$ is endowed with the orthonormal scalar product with respect to its canonical basis, denoted by $(,)$, and its normalized version $\left\langle  ,\right\rangle$, defined by
\begin{eqnarray}\label{eq2}
\left\langle   G,G^\prime \right\rangle :=\abs{\operatorname{Aut}(G)}
(G,G^\prime).
\end{eqnarray}

%

There is a linear isomorphism $I: \mathfrak{G}_{\leq n}\rightarrow\mathfrak{G}_{\leq n}$, the keystone to the theory, which maps a Gauss diagram of degree $n$ to the formal sum of its $2^n$ subdiagrams:
\begin{equation}\label{def:I}
I(G)=\sum_{\sigma\in \left\lbrace \pm 1\right\rbrace ^{n}}G_{(\sigma)},\end{equation}
where $G_{(\sigma)}$ is $G$ deprived from the arrows that $\sigma$ maps to $-1$ (see Definition\ref{def:GammaRm} for subdiagrams). 
The inverse map of $I$ is given by
\begin{equation}\label{def:I-1}
I^{-1}(G)=\sum_{\sigma\in \left\lbrace \pm 1\right\rbrace ^{n}} \operatorname{sign}(\sigma)G_{(\sigma)}.
\end{equation}
\begin{definition}
 A finite-type invariant for virtual knots in the sense of Goussarov-Polyak-Viro is a virtual knot invariant given by a \textit{Gauss diagram formula}
\begin{equation}\label{eq1}
\nu_\mathcal{G}: G \mapsto \left\langle  \mathcal{G},I(G)
\right\rangle,  
\end{equation}
where $\mathcal{G}\in \widehat{\mathfrak{G}}$. Such a formula \enquote{counts} the subdiagrams of $G$, with weights given by the coefficients of $\mathcal{G}$. Notice that only one of the two arguments of $\left\langle,
\right\rangle$ needs to be a finite sum for the expression to make sense. We do not make a distinction between a virtual knot invariant and the linear form induced on $\mathfrak{G}$.
\end{definition}

\subsubsection{The Polyak algebra}

A Gauss series $\mathcal{G}\in \widehat{\mathfrak{G}}$ defines a virtual knot invariant if and only if the function $\left\langle  \mathcal{G},I(.)\right\rangle$ is zero on the subspace spanned by $\mathrm{R}$-moves and $w$-moves relators. Hence one has to understand the image of that subspace under $I$ with a simple family of generators. This is the idea of the construction of the Polyak algebra (\citep{P1,GPV}).

\begin{figure}[h!]
\centering
\psfig{file=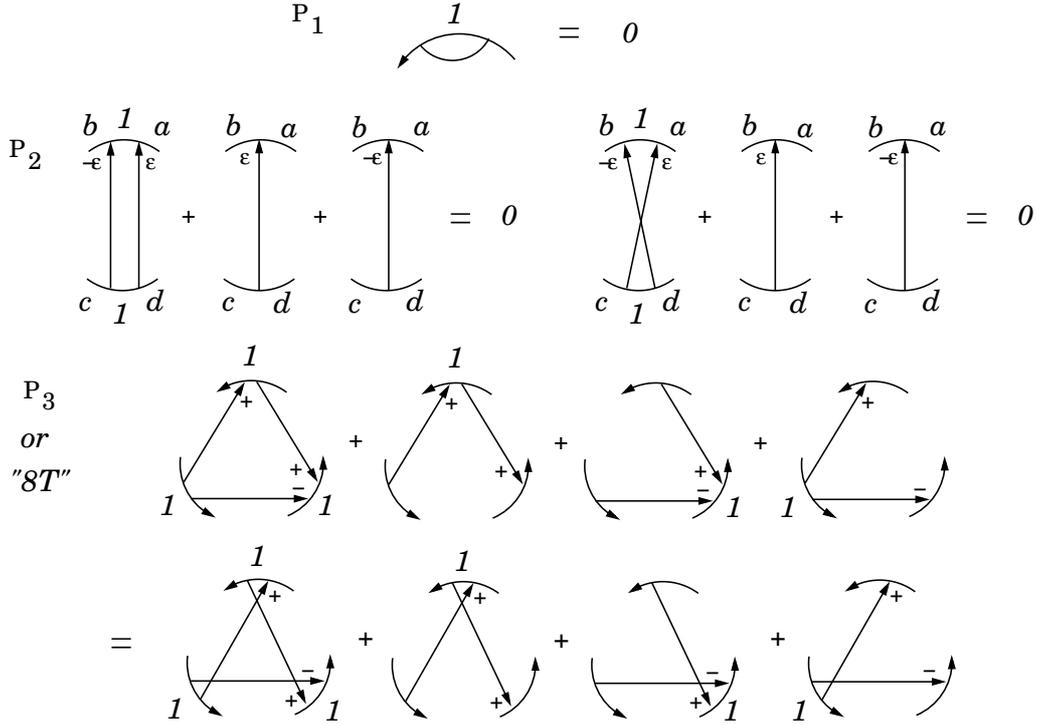,
scale=0.62}
\caption{The three kinds of Polyak relations -- only one $\mathrm{P}_3$ relation is shown, there is a second one obtained by reversing all the arrow orientations.}\label{pic:Pmoves}
\end{figure}

In the present case, $\mathcal{P}$ is defined as the quotient of $\mathfrak{G}$ by
\begin{itemize}
\item the relations shown in Fig.\ref{pic:Pmoves}, which we call $\mathrm{P}_1$, $\mathrm{P}_2$, $\mathrm{P}_3$ (or \textit{$8T$ relation}),
\item the $\mathrm{W}$ relation, which is simply the linear match of $w$-moves (\ie just replace the \enquote{$\leftrightsquigarrow$} with a \enquote{$=$} in all the relations from Fig.\ref{pic:conjugacy}).\end{itemize}

Be careful that unlike $\mathrm{R}_1$-moves, where an isolated arrow surrounding an edge marked with $1$ simply disappears, in a $\mathrm{P}_1$-move the presence of such an arrow completely kills the diagram. Fig.\ref{pic:Pmoves} does not feature the $\pi$-markings for $\mathrm{P}_3$ to lighten the picture, but they have to follow the usual \enquote{merge $\rightsquigarrow$ multiply} rule (see Definition~\ref{def:GammaRm}).

The following proposition extends Theorem $2.D$ from \citep{GPV}.

\begin{proposition}\label{thm:PolyakAlg}
The map $I$ induces an isomorphism $\mathfrak{G}/\RW
\rightarrow \mathfrak{G}/
\PW=:\mathcal{P}$. More precisely, $I$ induces an isomorphism between $\operatorname{Span}(\mathrm{R}_i)$ and $\operatorname{Span}(\mathrm{P}_i)$, for $i=1,2,3$, and between $\operatorname{Span}(\mathrm{W})$ and itself. It follows that the map $G\rightarrow I(G)\in \mathcal{P}$ defines a complete invariant for virtual knots.
\end{proposition}

\subsubsection{The symmetry-preserving injections}\label{sec:Sinj}
Depending on the context, one may have to consider simultaneously different types of Gauss diagrams, with more or less decorations. This subsection presents a natural way to do it, convenient from the viewpoint of Gauss diagram invariants. The construction requires one to choose a kind of combinatorial objects that is the \enquote{father} of all other kinds, in the sense of quotienting. We present the construction by taking as the father type that of Gauss diagrams on a group.

In first place, we do not regard Gauss diagrams up to homeomorphisms of the circle: the base circle is assumed to be the unit circle in $\CC$, the endpoints of the arrows are assumed to be located at the $2n$-th roots of unity, and the arrows are straight line segments. Such a diagram is called \textit{rigid}.

By a \enquote{type of rigid Gauss diagrams} we mean an equivalence relation on the set of rigid Gauss diagrams on $\pi$, which is required to satisfy two properties:\\
\textbf{1.} (Degree property) All diagrams in a given equivalence class shall have the same degree.\\
\textbf{2.} (Stability property) The action of $\ZZ/2n$ on the set of Gauss diagrams of degree $n$ shall induce an action on the set of degree $n$ equivalence classes.\\

Since every construction in this subsection is therefore destined to be homogeneous, the degree of all Gauss diagrams is once and for all set equal to $n$. A \textit{rigid Gauss diagram of type $\sim$} is an equivalence class under the relation $\sim$. A \textit{Gauss diagram (of type $\sim$)} is the orbit of a rigid diagram of type $\sim$ under the action of $\ZZ/2n$. The corresponding $\QQ$-spaces are respectively denoted by $\mathfrak{G}_\sim^\text{rigid}$ and $\mathfrak{G}_\sim$.

Since $\ZZ/2n$ is abelian, two elements from the same orbit have the same stabilizer, hence a Gauss diagram $G$ has a well-defined \textit{group of symmetries} $\operatorname{Aut}(G)$, which is the stabilizer of any of its rigid representatives under the action of $\ZZ/2n$.
Consequently, the space $\mathfrak{G}_\sim$ is endowed with a pairing $\left\langle  ,\right\rangle$ defined by \eqref{eq2}.

Now consider two types of rigid Gauss diagrams, say $1$ and $2$, such that relation $1$ is finer than relation $2$ (\enquote{$1\prec 2$}).

\begin{definition}[Forgetful projections]\label{def:proj}
A $1$-rigid diagram $G_1$ determines a unique $2$-rigid diagram whose $\ZZ/2n$-orbit only depends on that of $G_1$. This induces a natural surjective map at the level of Gauss diagram spaces, denoted by $$\operatorname{T_2^1}:\mathfrak{G}_{(1)}\twoheadrightarrow
\mathfrak{G}_{(2)}.$$
\end{definition}

Note that this map may be not well-defined on the spaces of formal series of Gauss diagrams, if some $2$-equivalence class contains infinitely many $1$-classes.

Example: the abelianization map $\operatorname{ab}$ (Definition~\ref{def:ab}) induces by linearity a forgetful projection from Gauss diagrams on $\pi$ to abelian diagrams on $\pi$, when $\pi$ is abelian.

\begin{definition}[Symmetry-preserving injections]\label{def:Sinj}
In the opposite way, there is a map $\mathfrak{G}_{(2)}^\text{rigid}
\rightarrow
\mathfrak{G}_{(1)}^\text{rigid}$ that sends a $2$-rigid diagram $G_2$ to the formal sum of all $1$-classes that it contains. When this sum is pushed in $\mathfrak{G}_{(1)}$, the result
\begin{itemize}
\item is well-defined: a $2$-rigid diagram cannot contain infinitely many rigid representatives of a given Gauss diagram of type $1$, since the orbits are finite ($\ZZ/2n$ is finite).
\item only depends on the $\ZZ/2n$-orbit of $G_2$. 
\end{itemize}
This induces an injective \textit{symmetry-preserving map}  at the level of formal series,
$$\operatorname{S_2^1}:\widehat{\mathfrak{G}_{(2)}}
\hookrightarrow
\widehat{\mathfrak{G}_{(1)}}.$$
$\operatorname{S_2^1}$ is well-defined, componentwise, since $2$-rigid diagrams from different $\ZZ/2n$-orbits contain $1$-rigid diagrams from disjoint sets of $\ZZ/2n$-orbits (the images of two different Gauss diagrams do not overlap). It is injective for the same reason. 
\end{definition}

The terminology is explained by the following fundamental formula:

\begin{lemma}\label{Sfund}
With notations as above, for any Gauss diagram $G_2$ of type $2$,
\begin{equation}
\S_2^1(G_2)=\sum_{\operatorname{T_2^1}\left( G_1\right)=G_2} \frac{\abs{\operatorname{Aut}(G_2)}}{\abs{\operatorname{Aut}(G_1)}}G_1.
\end{equation}
\end{lemma}

\textit{Informally, the weight given to a preimage of $G_2$ under $\operatorname{T_2^1}$ is the amount of symmetry that it has lost by the gain of more information. Note that the weights are integers, since $\operatorname{Aut}(G_1)$ identifies with a subgroup of $\operatorname{Aut}(G_2)$.}

\begin{proof}
Fix a representative $G_2^\text{rigid}$ of $G_2$. By the stability property, $\operatorname{Aut}(G_2)$ acts on the set of $1$-classes contained in $G_2^\text{rigid}$. Moreover, by definition of $\operatorname{Aut}(G_2)$, two different orbits under that action still lie in different orbits under the action of $\ZZ/2n$ itself. Therefore there is a $1-1$ correspondence between the $\operatorname{Aut}(G_2)$-orbits and the Gauss diagrams that happen in the sum $\S_2^1(G_2)$. The stabilizer of a given $1$-class $G_1^\text{rigid}$ is by definition $\operatorname{Aut}(G_1)$, whence the cardinality of the corresponding orbit, which is also the coefficient of $G_1$ in $\S_2^1(G_2)$, is $\frac{\abs{\operatorname{Aut}(G_2)}}{\abs{\operatorname{Aut}(G_1)}}$.
\end{proof}

\begin{proposition}\label{prop:ST}
\ 
\begin{enumerate}

\item For any three relations such that $1\prec 2\prec 3$, the following diagrams commute:
\begin{center}
\begin{tikzpicture}
\node (0) at (0.45,0.45) {$\circlearrowleft$};
\node (.) at (5.5,0) {.};
\node (1) at (0,0) {$\mathfrak{G}_{(2)}$};
\node (2) at (1.5,0) {$\mathfrak{G}_{(3)}$};
\node (3) at (0,1.5) {$\mathfrak{G}_{(1)}$};
\draw[->>,>=latex] (1) -- (2) node[midway,below] {$\operatorname{T_3^2}$};
\draw[->>,>=latex] (3) -- (1) node[midway,left] {$\operatorname{T_2^1}$};
\draw[->>,>=latex] (3) -- (2) node[midway,above right] {$\operatorname{T_3^1}$};

\node (10) at (3,1.5) {$\widehat{\mathfrak{G}_{(3)}}$};
\node (20) at (4.5,1.5) {$\widehat{\mathfrak{G}_{(2)}}$};
\node (30) at (4.5,0) {$\widehat{\mathfrak{G}_{(1)}}$};
\node (40) at (4.05,1.05) {$\circlearrowleft$};
\draw[right hook->,>=latex] (10) -- (20) node[midway,above] {$\S_3^2$};
\draw[right hook->,>=latex] (20) -- (30) node[midway,right] {$\S_2^1$};
\draw[right hook->,>=latex] (10) -- (30) node[midway,below left] {$\S_3^1$};

\end{tikzpicture}
\end{center}

\item Injections and projections are pairwise $\left\langle  ,\right\rangle$-adjoint, in the sense that 
$$
\begin{array}{cc}
\forall \,\, \mathcal{G}_1\in
\mathfrak{G}_{(1)},\, \mathcal{G}_2\in\widehat{
\mathfrak{G}_{(2)}},& \left\langle  \S_2^1\left(\mathcal{G}_2
\right),\mathcal{G}_1\right\rangle =\left\langle  \mathcal{G}_2,
\T_2^1\left(
\mathcal{G}_1\right)\right\rangle.
\end{array}
$$
\item 
$\operatorname{Im} \S_2^1 = \operatorname{Ker}^\bot \operatorname{T_2^1}.
$

\end{enumerate}
\end{proposition}

\begin{proof}

$1.$ The first diagram commutes directly from the definition of the maps $\T_i^j$. As for the maps $\S_i^j$, since they are defined componentwise it is enough to check it for a single diagram $G_3$. In that case, it is a consequence of Lemma~\ref{Sfund} and the relation $\operatorname{T_3^1}=\operatorname{T_3^2}\circ \operatorname{T_2^1}.$

$2.$ In both sides, it is clear that only a finite number of terms in $\mathcal{G}_2$ are relevant, namely those that are projections of some terms of $\mathcal{G}_1$ under $\operatorname{T_2^1}$. Thus, by bilinearity, it is enough to consider single diagrams $G_1$ and $G_2$. If $G_2\neq \operatorname{T_2^1}(G_1)$, then both sides are $0$. 
If $G_2= \operatorname{T_2^1}(G_1)$, then $$\begin{array}{ccl}
\left\langle  \S_2^1\left(G_2\right),G_1\right\rangle & = & 
\left\langle  
\frac{\abs{\operatorname{Aut}(G_2)}}{\abs{\operatorname{Aut}(G_1)}}
G_1
,G_1\right\rangle \\
 & = & \abs{\operatorname{Aut}(G_2)},
\end{array}$$
while
$$\begin{array}{ccl}
\left\langle   G_2,
\T_2^1\left(
G_1\right)\right\rangle & = & \left\langle   G_2,G_2\right\rangle\\
& = & \abs{\operatorname{Aut}(G_2)}.\end{array}
$$

$3.$ The inclusion $\operatorname{Im} \S_2^1 \subset \operatorname{Ker}^\bot \operatorname{T_2^1}$ follows immediately from $2.$ For the converse, pick a Gauss diagram series $\mathcal{G}_1$ in $\operatorname{Ker}^\bot \operatorname{T_2^1}$. For any two $2$-related Gauss diagrams of type $1$, $G_1$ and $G_1^\prime$, one has
$$\left\langle   \mathcal{G}_1,G_1-G_1^\prime \right\rangle=0.$$
Thus, if $G_2$ is a Gauss diagram of type $2$, one can define $\phi(G_2)$ to be the value of $\left\langle   \mathcal{G}_1,G_1\right\rangle$ for any preimage $G_1$ of $G_2$ under $\operatorname{T_2^1}$, and set
$$\mathcal{G}_2=\sum \frac{\phi(G_2)}{\abs{\operatorname{Aut}(G_2)}}G_2,$$
where the sum runs over all Gauss diagrams of type $2$. Finally,
$$
\begin{array}{ccl}
\S_2^1(\mathcal{G}_2) & = & \sum \frac{\abs{\operatorname{Aut}(\operatorname{T_2^1}\left( G_1\right))}}{\abs{\operatorname{Aut}(G_1)}}
\frac{\phi(\operatorname{T_2^1}\left( G_1\right))}{\abs{\operatorname{Aut}(\operatorname{T_2^1}\left( G_1\right))}}G_1\\
 & = & \sum \frac{\left\langle  \mathcal{G}_1, G_1\right\rangle}{\abs{\operatorname{Aut}(G_1)}} G_1\\
 & = & \mathcal{G}_1
\end{array}.$$

\end{proof}

In practice, point $3$ is useful in both directions: whether one needs a characterization of the series that lie in the image of some map $\S$ (Lemma~\ref{ImS}), or of the series that define invariants under some kind of moves (Propoosition~\ref{prop:omega}). Point $2$ states that symmetry-preserving maps are the good dictionary to understand invariants that were defined via forgetful projections.

\begin{remark}\label{openbar}
Every construction and result in this subsection can be repeated by replacing the set of rigid Gauss diagrams of degree $n$ with any set endowed with the action of an abelian finite group.
\end{remark}

\subsubsection{Arrow diagrams and homogeneous invariants}\label{sec:arrows}

\begin{definition}[see \citep{P1,PolyakViro}]\label{def:arrowdiag}
An \textit{arrow diagram (on $\pi$)} is a Gauss diagram $G$ (on $\pi$) of which the signs decorating the arrows have been forgotten. As usual, it is considered up to homeomorphisms of the circle.\end{definition}

Arrow diagram spaces $\mathfrak{A}_{n}$, $\mathfrak{A}_{\leq n}$, $\mathfrak{A}$, the hat versions, and the pairings $(,)$ and $\left\langle  ,\right\rangle$ are defined similarly to their signed versions (Subsection \ref{sec:GDspaces}). We use notations $A$
for an arrow diagram and $\mathcal{A}$ for an \textit{arrow diagram series} -- \ie an element of $\widehat{\mathfrak{A}}$.

\subsubsection*{Arrow diagram formulas}

In the language of Subsection~\ref{sec:Sinj}, arrow diagrams are a kind of Gauss diagrams satisfying the degree and stability properties -- the equivalence relation on rigid Gauss diagrams is given by $G\sim G^\prime \Leftrightarrow$ one may pass from $G$ to $G^\prime$ by writhe changes. Therefore, an arrow diagram $A$ has a well-defined symmetry group $\operatorname{Aut}(A)$, and there are a symmetry-preserving map $\S_a:\widehat{\mathfrak{A}}\hookrightarrow \widehat{\mathfrak{G}}$ and a projection $\T_a:\mathfrak{G}\twoheadrightarrow\mathfrak{A}$. 
However, for the purpose of defining arrow diagram invariants, we are going to twist these maps a little, by pushing an additional sign into the weights:

\begin{definition}\label{twistedS}
Define the linear maps $S:\widehat{\mathfrak{A}}
\rightarrow
\widehat{\mathfrak{G}}$
and $T:\mathfrak{G}\rightarrow
\mathfrak{A}$ componentwise by
\begin{eqnarray}
S(A) & := & \sum_{\T_a(G)=A}\operatorname{sign}(G)\frac{\operatorname{Aut}(A)}{\operatorname{Aut}(G)}G
\\
T(G) & := & \operatorname{sign}(G)\T_a(G)\label{defT}
\end{eqnarray}
\end{definition}

\begin{proposition}\label{lem:ST} \ 
\begin{enumerate}
\item $S$ and $T$ are $\left\langle  ,\right\rangle$-adjoint, in the sense that 
$$
\begin{array}{cc}
\forall \,\, \mathcal{G}\in
\mathfrak{G},\, \mathcal{A}\in\widehat{
\mathfrak{A}},& \left\langle   S\left(\mathcal{A}
\right),\mathcal{G}\right\rangle =\left\langle  \mathcal{A},
T\left(
\mathcal{G}\right)\right\rangle.
\end{array}
$$
\item 
$\operatorname{Im} S = \operatorname{Ker}^\bot T.$

\end{enumerate}
\end{proposition}
The proof is completely similar to that of Proposition~\ref{prop:ST}.

\begin{definition}
A Gauss diagram formula that lies in the image of the map $S$ is called an \textit{arrow diagram formula}.
\end{definition}

\subsubsection*{Homogeneous invariants}
\begin{definition}
For each $n\in \NN$, there is an orthogonal projection $p_n : \widehat{\mathfrak{G}}\rightarrow \widehat{\mathfrak{G}}_n$ with respect to the scalar product $\left\langle  ,\right\rangle$. For $\mathcal{G}\in \widehat{\mathfrak{G}}$, the \textit{principal part} of $\mathcal{G}$ is defined by $p_n (\mathcal{G})$, with $n=\operatorname{deg}(G)$.
$\mathcal{G}$ is called \textit{homogeneous} if it is equal to its principal part.
\end{definition}

\begin{figure}[h!]
\centering 
\psfig{file=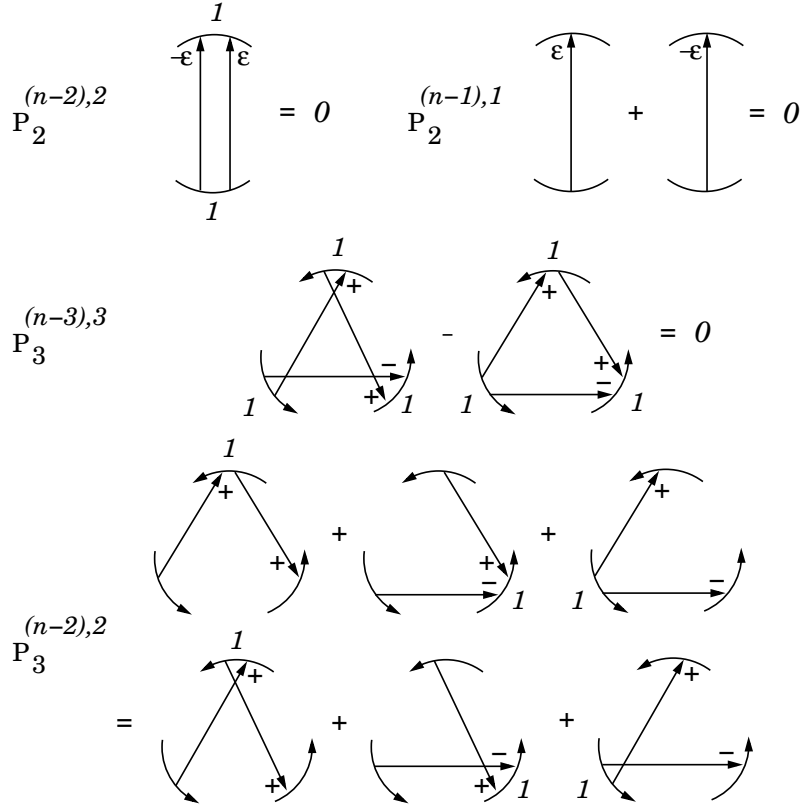,
scale=0.65}
\caption{Some homogeneous Polyak relations}\label{pic:homogeneous}
\end{figure}

\begin{figure}[h!]
\centering 
\psfig{file=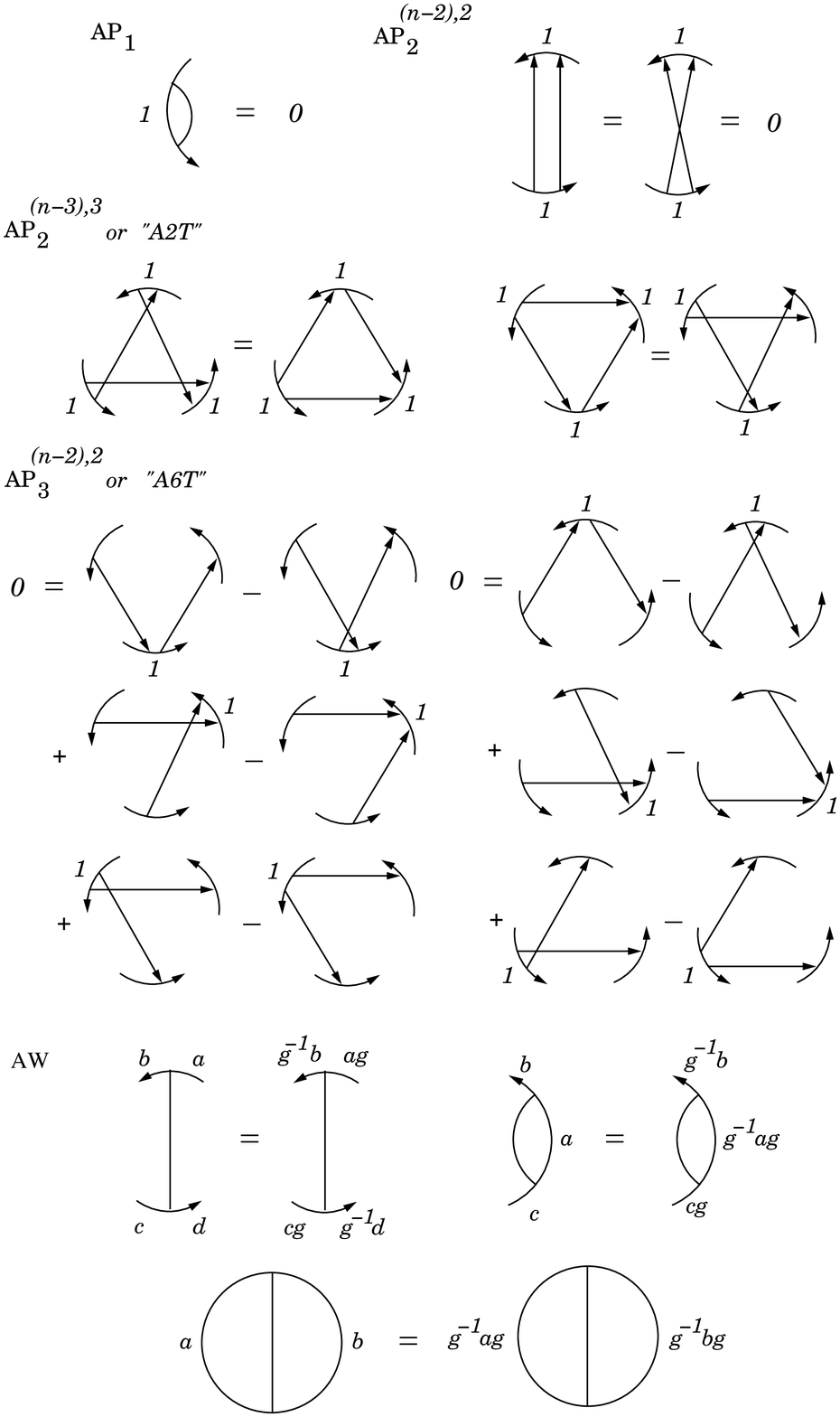,
scale=0.64}
\caption{The homogeneous arrow relations -- as usual, in $\mathrm{AW}$ the orientation of an arrow is reversed if and only if $w(g)=-1$.}\label{pic:Arelations}
\end{figure}

Let $\mathcal{G}$ be a homogeneous Gauss series. Then $\mathcal{G}$ satisfies the $\mathrm{P}_2$ and $\mathrm{P}_3$ relations (\ie $\left\langle  \mathcal{G},\mathrm{P}_i\right\rangle=0$ for $i=2,3$) if and only if it satisfies the \textit{homogeneous} relations $\left\langle  \mathcal{G},p_k(\mathrm{P}_i)\right\rangle=0$ for all $k$. These are denoted by  $\mathrm{P}_2^{(n-1),1}$, $\mathrm{P}_2^{(n-2),2}$, $\mathrm{P}_3^{(n-2),2}$ (or $G6T$) and $\mathrm{P}_3^{(n-3),3}$ (or $G2T$). The parenthesized numbers in exponent indicate in each case how many arrows are unseen. $\mathrm{P}_1$ and $\mathrm{W}$ relations are already homogeneous and do not get a new name. Some examples are shown on Fig.\ref{pic:homogeneous}; for a full list, just consider the projections of the relations of Fig.\ref{pic:Pmoves}.

Homogeneous relations are also defined for arrow diagram spaces, denoted by 
$\mathrm{AP}_1$,
$\mathrm{AP}_2^{(n-2),2}$, $\mathrm{AP}_3^{(n-2),2}$ (or $A6T$), $\mathrm{AP}_3^{(n-3),3}$ (or $A2T$) and $\mathrm{AW}$ (Lemma~\ref{ImS} below explains why $\mathrm{AP}_2^{(n-1),1}$ is useless: it reads $0=0$). They are the images under $T$ \eqref{defT} of the homogeneous relations for Gauss diagrams -- in particular one should be especially careful to the signs in the $A6T$ relations. A full list is presented in Fig.\ref{pic:Arelations}.

\begin{lemma}\label{SpanOrth}
Let $\mathcal{A}\in \widehat{\mathfrak{A}}$ and let $\mathrm{X}$ be a name among
$\mathrm{P}_1$, $\mathrm{P}_2^{(n-2),2}$, $\mathrm{P}_3^{(n-2),2}$, $\mathrm{P}_3^{(n-3),3}, \mathrm{W}$. Then
$$\mathcal{A}\in \operatorname{Span}^\perp(\mathrm{AX})\Longleftrightarrow S(\mathcal{A})\in \operatorname{Span}^\perp(\mathrm{X}).$$
where orthogonality is as usual in the sense of $\left\langle  ,\right\rangle$.
\end{lemma}
\begin{proof}
It is a direct consequence of part $1.$ of Proposition~\ref{lem:ST}.
\end{proof}

\begin{lemma}\label{ImS}
Let $\mathcal{G}\in
\widehat{\mathfrak{G}}$. Then $\mathcal{G}$ lies in the image of the map $S:\widehat{\mathfrak{A}}\rightarrow
\widehat{\mathfrak{G}}$ if and only if $\mathcal{G}$ satisfies all the homogeneous relations $\left\langle  \mathcal{G},\mathrm{P}_2^{(n-1),1}\right\rangle=0$.
\end{lemma}

\begin{proof}
Notice that the $\mathrm{P}_2^{(n-1),1}$ relators span the kernel of the map $T$. Hence the result follows from point $2.$ of Proposition~\ref{lem:ST}.
\end{proof}

The following are proved in a particular case in \citep{MortierPolyakEquations} (Lemma~3.2 and Theorem~2.5); the proof can be readily adapted to the present situation.

\begin{lemma}\label{crucial2T6T} For all $n\geq 3$:
$$\operatorname{Span}(A2T) \subseteq \operatorname{Span}(A6T)\oplus \operatorname{Span}(\mathrm{AP}_2^{(n-2),2}).$$
\end{lemma}

\begin{theorem}\label{thm:ArrowHomogeneous}
Arrow diagram formulas are exactly the linear combinations of homogeneous Gauss diagram formulas.
\end{theorem}

\subsubsection{Based and degenerate diagrams}\label{subsec:bas_deg}

\begin{definition}\label{def:basdeg}
A \textit{based} Gauss diagram is a Gauss diagram together with a distinguished (\textit{base}) edge. Based arrow diagrams are defined similarly. The corresponding spaces are denoted by $\mathfrak{G}_\bullet$ and $\mathfrak{A}_\bullet$, in reference to the dot that we use in practice to pinpoint the distinguished edge.

A \textit{degenerate Gauss diagram (with one degeneracy)} is a Gauss diagram in which one edge, whose endpoints belonged to two different arrows, has been shrunk to a point. The spaces of degenerate diagrams are denoted by $\mathfrak{DG}$ and $\mathfrak{DA}$ respectively.
\end{definition}

Even though they would encode long knots in the classical theory, based diagrams have only a combinatorial interest here. On the other hand, degenerate diagrams have a natural topological interpretation that is explained in \citep{FT1cocycles}.

The space of degenerate arrow diagrams is meant to be quotiented by the so-called \textit{triangle relations}, shown in Fig.\ref{pic:triangle}. The quotient space is denoted by $\mathfrak{DA}/\nab$. These relations originated in the early work of M.Polyak on arrow diagrams (\citep{PolyakTalk, P1}, see also \cite{PVCasson}).

\begin{figure}[h!]
\centering 
\psfig{file=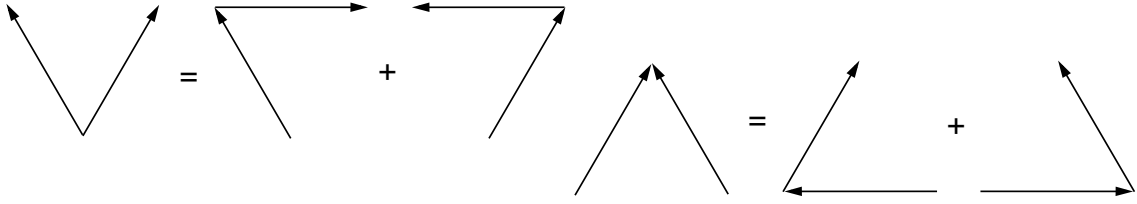,scale=0.67}
\caption{The triangle relations}\label{pic:triangle}
\end{figure}

\begin{definition}\label{def:monotonic}
Call a degenerate diagram \textit{monotonic} if an arrowhead and an arrowtail meet at the degenerate point.
\end{definition}

\begin{lemma}
$\widehat{\mathfrak{D}\mathfrak{A}/
\nab}$ is naturally isomorphic with the $\QQ$-space of formal series of monotonic arrow diagrams.
\end{lemma}

\begin{proof}
It suffices to show that the set of monotonic diagrams forms a basis of $\mathfrak{DA}/\nab$. It is clearly a generating set thanks to the $\nabla$ relations, and it is free because every non monotonic diagram happens in exactly one relation, and every relation contains exactly one of them.
\end{proof}

\subsection{Invariance criteria}\label{sec:polyak}

\subsubsection*{$w$-invariance}
\begin{proposition}
\label{prop:omega}
There is an injective \enquote{symmetry-preserving} map
$\S_{aw}^a:\widehat{\mathfrak{A}/\AW}\hookrightarrow
\widehat{\mathfrak{A}}$ defined componentwise by the formula
$$\alpha\mapsto\sum_{A\in\alpha}
\frac{\abs{\operatorname{Aut}(\alpha)}}{\abs{\operatorname{Aut}(A)}}A.$$
If $\mathcal{A}\in\widehat{\mathfrak{A}}$, then the map $G\mapsto  \left\langle \! \left\langle \mathcal{A},G \right\rangle\!\right\rangle= \left\langle  S(\mathcal{A}),I(G)\right\rangle$ is invariant under $w$-moves if and only if $\mathcal{A}$ lies in the image of $\S_{aw}^a$.
\end{proposition}
\begin{proof}
The equivalence relation defined by $\mathrm{AW}$-moves and writhe changes on the set of rigid Gauss diagrams on $\pi$ satisfies the degree and stability properties from Subsection~\ref{sec:Sinj}. Hence one may apply the results of that section to get the existence and elementary properties of $\S_{aw}^a$. The last assertion follows from sucessive application of the $\mathrm{W}$ part of Lemma~\ref{SpanOrth} and point $3$ of Proposition~\ref{prop:ST}.
\end{proof}

This means that an arrow diagram formula must be represented by a series of $w$-orbits of arrow diagrams. In practice, this condition is most of time satisfied by construction. An important example is the formal sum of all elements contained in a set that is stable under $w$-moves.

\subsubsection*{$\mathrm{R}_1$ and $\mathrm{R}_2$ invariance}

\begin{proposition}
\label{thm:R1et2}
Let $\mathcal{A}\in \widehat{\mathfrak{A}}$. Then the function $G \mapsto\left\langle  S(\mathcal{A}),I(G)\right\rangle$ is invariant under $\mathrm{R}_1$ (\resp $\mathrm{R}_2$) moves if and only if $\mathcal{A}$ satisfies all $\mathrm{AP}_1$ (\resp $\mathrm{AP}_2^{
(n-2),2}$) relations. \end{proposition}

\begin{proof}
By Proposition~\ref{thm:PolyakAlg}, the $\mathrm{R}_1$ and $\mathrm{R}_2$ invariance are equivalent to the relations $\left\langle  S(\mathcal{A}),\mathrm{P}_1\right\rangle=0$ and $\left\langle  S(\mathcal{A}),\mathrm{P}_2\right\rangle=0$ being satisfied. Lemma~\ref{ImS} implies that the second of these relations is actually equivalent to $\langle   S(\mathcal{A}),\mathrm{AP}_2^{(n-2),2}
\rangle =0$. Lemma~\ref{SpanOrth} concludes the proof.
\end{proof}

In practice, these conditions are easy to check naked-eye.

\subsubsection*{$\mathrm{R}_3$ invariance}

\begin{definition}\label{def:nice}
Say that a based diagram $A_\bullet$ is \textit{nice} if its base edge
\begin{itemize}
\item is decorated by $1$, 
\item is bounded by the endpoints of two different arrows. 
\end{itemize}
\end{definition}

\begin{lemma}\label{epsilonbased}
Nice based diagrams have a sign $\varepsilon$ induced by Definition~\ref{def_epsilon}.
\end{lemma}
\begin{proof}
The sign $\varepsilon$ from Definition~\ref{def_epsilon} takes as an argument the edge of a Gauss diagram. Since it does not depend on the writhes of the arrows, it is well-defined for arrow diagrams with a preferred edge.
\end{proof}

\begin{definition}\label{def:delta}
Let $A_\bullet$ be a based arrow diagram. If $A_\bullet$ is not nice, then put $\delta (A_\bullet)=0.$
If $A_\bullet$ is nice, then
\begin{enumerate}
\item Shrink the base edge to a point.
\item Multiply the resulting degenerate diagram by $\varepsilon(A_\bullet)$, and call the result $\delta (A_\bullet).$

\end{enumerate}
This process defines a map $\delta:\widehat{\mathfrak{A}_\bullet} \rightarrow \widehat{\mathfrak{D}
\mathfrak{A}}$ -- well defined since any monotonic diagram has finitely many preimages.
Now let $A$ be an arrow diagram, and denote by $\bullet(A)\in \mathfrak{A}_\bullet$ the sum of all based diagrams that one can form by choosing a base edge in $A$. Finally, define
$$
\begin{array}{cccc}
d: & \widehat{\mathfrak{A}} & \rightarrow & \widehat{\mathfrak{DA}} \\
 & \mathcal{A} & \mapsto & 
\delta (\bullet(\mathcal{A}))
\\
\end{array}.
$$
\end{definition}

\begin{theorem}\label{thm:main}
Let $\mathcal{A}\in \widehat{\mathfrak{A}}$ satisfy the $\mathrm{R}_2$ invariance condition from Proposition~\ref{thm:R1et2}. Then the following are equivalent:
\begin{itemize}
\item The map $G\mapsto\left\langle  S(\mathcal{A}),I(G)\right\rangle$ is invariant under $\mathrm{R}_3$-moves.
\item $d(\mathcal{A})=0$ modulo the triangle relations.
\item $\mathcal{A}\in \operatorname{Span}^\perp (A6T)$.
\end{itemize}
\end{theorem}

The proof is similar to that of Theorem~3.6 from \citep{MortierPolyakEquations}.

\subsubsection{Invariance criterion for $w$-orbits}

As we have seen, an arrow diagram series defines an invariant only if it has a preimage in $\widehat{\mathfrak{A}/\AW}$. The above criteria for invariance under $\mathrm{R}$-moves nicely extend in terms of that preimage. This is especially interesting when $w$-orbits are well understood, for instance if $\pi$ is abelian and $w$ is trivial (Proposition~\ref{prop:abelian}).

\subsubsection*{$w$-moves for based and degenerate diagrams}
These moves are defined similarly to the regular version (Fig.\ref{pic:conjugacy}, top-left and top-right), with additional moves in the degenerate case, that modify the neighborhood of the arrows meeting at the degenerate point (Fig.\ref{pic:degwmove} shows the extreme cases -- there are obvious intermediate ones, when only one or two of the unseen arcs is empty). The arrows both change orientations if $w(g)=-1$, and keep the same otherwise.

\begin{figure}[h!]
\centering 
\psfig{file=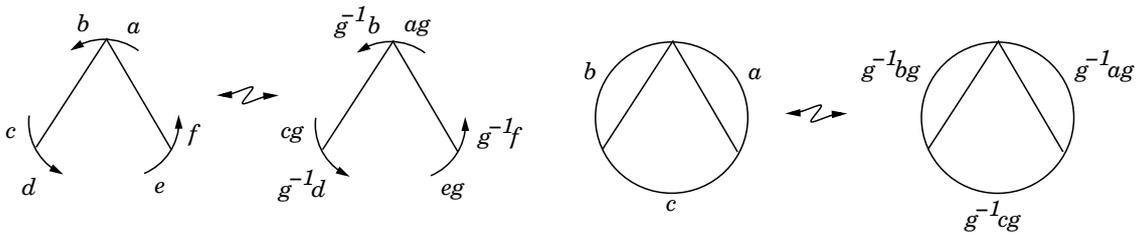,scale=0.67}
\caption{The \enquote{degenerate} $w$-move -- the most general and the most exceptional cases. The two arrows change orientations if and only if $w(g)=-1$.}\label{pic:degwmove} \end{figure}

\begin{definition}
Pick a based arrow diagram $A_\bullet$. If its base edge is bounded by twice the same arrow, then set $\delta_w (\left[ A_\bullet\right] )=0$. If it is bounded by two different arrows, then
\begin{enumerate}
\item pick a nice diagram $A_\bullet^{(1)}$ $\mathrm{AW}$-equivalent to $A_\bullet$,

\item set $\delta_w (\left[ A_\bullet\right] )=\left[ \delta(A_\bullet^{(1)})\right] $.
\end{enumerate}

Finally, set 
$$d_w(\left[ A\right] )=\delta_w (\left[\bullet(A) \right] ).$$
As usual, $\delta_w$ and $d_w$ are defined componentwise on formal series of $w$-orbits.

\end{definition}

\begin{proof}[Consistency of the definition]

Observe that when the endpoints of an edge belong to two different arrows, then any value can be given to its marking by using the appropriate $w$-move. This proves that step $1$ is always possible -- though not in a unique way.

For step $2$, first notice that two $w$-moves performed on different arrows from an arrow diagram commute, so that any finite sequence of $w$-moves amounts to a sequence made of one move for each arrow. If such a sequence leaves the marking of the base edge unchanged (and equal to $1$), then the moves on the two adjacent arrows must involve the same conjugating element $g$.
It follows that whenever $A_\bullet$ and $A^\prime_\bullet$ are nice and lie in the same $w$-orbit, 
\begin{itemize}
\item $\varepsilon(A_\bullet)=
\varepsilon(A^\prime_\bullet),$
\item the degenerate diagrams obtained by shrinking their bases also lie in the same orbit.
\end{itemize} 
Hence $\delta_w$ is well-defined.
\end{proof}

\subsubsection*{How to handle the quotients by $\nabla$ relations}

Note that the quotient of $\mathfrak{DA}$ by the $\nabla$ relations does not fit in the general framework in which symmetry preserving maps were introduced: indeed, it does not come from an equivalence relation at the level of the \textit{set} of diagrams. 

However, the set of classes of monotonic diagrams forms a basis of $\mathfrak{DA}/\nab$. This induces an injective section $i$ of the projection $s:\mathfrak{DA}\twoheadrightarrow\mathfrak{DA}/\nab$. Both $s$ and $i$ extend componentwise to formal series of $w$-orbits.

The same phenomenon happens between $\widehat{\mathfrak{DA}/\AW}$ and $\widehat{\mathfrak{DA}/\AWnab}$, because $w$-moves never change the status of a degenerate diagram -- monotonic or not -- and the set of monotonic $w$-orbits still forms a basis of $\mathfrak{DA}/\AWnab$.
Again there are maps $s_w$ and $i_w$ such that $s_w\circ i_w=\operatorname{Id}$, at the level of formal series.

Finally, this allows us to construct a symmetry-preserving map $\mathrm{S}_{w\nabla}^{\nabla}:\widehat{\mathfrak{DA}/\AWnab}\rightarrow \widehat{\mathfrak{DA}/\nab}$ , by regarding the restriction of $\mathrm{S}_{dw}^{d}: \widehat{\mathfrak{DA}/\AW}\rightarrow \widehat{\mathfrak{DA}}$ to the subspaces of monotonic diagrams.

\subsubsection*{Summary}
In the following diagram, the two squares and the two triangles on the left are commutative, as well as the internal and external squares on the right. Except for $\mathrm{S}_{w\nabla}^{\nabla}$,
all vertical arrows are symmetry-preserving injections in the usual sense.

\begin{center}
\begin{tikzpicture}[baseline=(current bounding box.center)]
\node (.) at (4,-0.8) {.};
\node (1) at (-5,2) {$\widehat{\mathfrak{A}}$};
\node (1b) at (-2.5,0.8) {$\widehat{\mathfrak{A}_\bullet}$};

\node (2) at (0,2) {$\widehat{\mathfrak{DA}}$};

\node (3) at (-5,-2) {$\widehat{\mathfrak{A}/\AW}
$};
\node (3b) at (-2.5,-0.8) {$\widehat{\mathfrak{A}_\bullet/\AW}
$};
\node (4) at (0,-2) {$\widehat{\mathfrak{DA}/\AW}$};
\node (6) at (2.6,-0.8) {$\widehat{\mathfrak{DA}/\AWnab}$};

\node (5) at (2.6,0.8) {$\widehat{\mathfrak{DA}/\nab}$};

\node (s) at (1.95,2.15) {$i$};

\node (sw) at (2.1,-2.2) {$i_w$};

\draw[->,>=latex] (1) -- (2) node[midway,above] {$d$};
\draw[->,>=latex] (1) -- (1b) node[midway,below] {$\bullet$};
\draw[->,>=latex] (1b) -- (2) node[midway,below] {$\delta$};

\draw[->,>=latex] (3) -- (3b) node[midway,above] {$\bullet$};
\draw[->,>=latex] (3b) -- (4) node[midway,above] {$\delta_w$};

\draw[right hook->,>=latex] (3) -- (1) node[midway,left] {$\mathrm{S}_{aw}^a$};
\draw[right hook->,>=latex] (3b) -- (1b) node[midway,left] {$\mathrm{S}_{\bullet w}^\bullet$};
\draw[right hook->,>=latex] (4) -- (2) node[midway,left] {$\mathrm{S}_{dw}^{d}$};
\draw[->,>=latex] (3) -- (4) node[midway,below] {$d_w$};
\draw[->>,>=latex] (4) -- (6) node[midway,above left] {$s_w$};
\draw[left hook->,>=latex] (6) to[out=245,in=350] (4);

\draw[->>,>=latex] (2) -- (5) node[midway,below left] {$s$};
\draw[right hook->,>=latex] (5) to[out=120,in=15] (2);

\draw[right hook->,>=latex] (6) -- (5) node[midway,right] {$\mathrm{S}_{w\nabla}^{\nabla}$};

\end{tikzpicture}
\end{center}

\begin{theorem}
\label{thm:mainorbits}
An arrow diagram series $\mathcal{A}$ is an arrow diagram formula if and only if each of the following holds:
\begin{enumerate}
\item $\mathcal{A}$ has a preimage $\mathcal{A}_w$ by $\mathrm{S}_{aw}^a$.
\item $\mathcal{A}_w$ is mapped to $0$ in $\widehat{\mathfrak{DA}/\AWnab}$.
\item $\mathcal{A}_w$ satisfies the equations
$$\begin{array}{ccc}

\left\langle  \mathcal{A}_w,\mathrm{T}_{aw}^a (\mathrm{AP_1})\right\rangle=0 & \text{and} & \left\langle  \mathcal{A}_w,\mathrm{T}_{aw}^a (\mathrm{AP_2^{(n-2),2}})\right\rangle=0
\end{array}.$$
\end{enumerate}

\end{theorem}

\begin{proof}
$1$ is necessary because of Proposition~\ref
{prop:omega}, and $3$ because of Proposition~\ref{thm:R1et2} and point $2$ of Proposition~\ref{prop:ST}.
If $1$ and $3$ are satisfied, then Theorem~\ref{thm:main} implies that $\mathcal{A}$ defines an arrow diagram formula if and only if $s(d(\mathcal{A}))=0$. This is equivalent to $\mathrm{S}_{w\nabla}^{\nabla}( s_w( d_w(\mathcal{A}_w)))=0$ since the diagram commutes, and to $s_w( d_w(\mathcal{A}_w)))=0$ since $\mathrm{S}_{w\nabla}^{\nabla}$ is injective.
\end{proof}

\begin{remark}
To apply the above theorem to an element of $\widehat{\mathfrak{A}/\AW}$, one never needs to push it through symmetry-preserving maps (in the upper half of the diargam). Hence the checkings are done using the most compact expressions.
Also, Point $3$ can be checked separately on each $w$-orbit that happens in $\mathcal{A}_w$, and for each of them it can be checked on any representative diagram by a simple criterion.\newline
For $\mathrm{AP}_1$ relations to hold, the situation at the left of Fig.\ref{R1R2orbits} simply must never happen (the $1$-marking is invariant under $w$-moves).\newline
For $\mathrm{AP_2^{(n-2),2}}$, the situations at the middle and at the right of the picture are forbidden:
\begin{enumerate}
\item if $w(g)=1$ and the arrows have \enquote{the same} orientation (in the sense of the picture),
\item if $w(g)\neq 1$ and the arrows have \enquote{different} orientations.
\end{enumerate}
Again these conditions are stable by $w$-moves.
\end{remark}

\begin{figure}[h!]
\centering 
\psfig{file=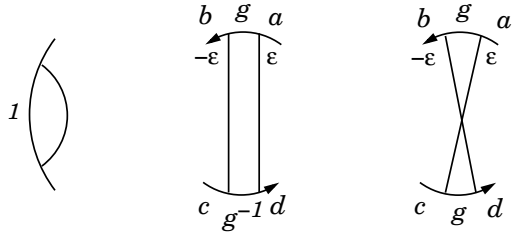,scale=0.67}
\caption{Forbidden situations -- the rules for the arrow orientations are explained above.}\label{R1R2orbits} \end{figure}

\subsection{Examples and applications}\label{sec:Grishanov}

It was noticed by M. Polyak \citep{PolyakTalk} that several families of formulas describing the finite-type invariants extracted from the Conway polynomial \citep{ChmutovKhouryRossi, Brandenbursky} actually define invariants of virtual knots. That also led him to conjecture that there should be an invariance criterion such as the one we give here for $\mathrm{R}_3$-moves.

We describe here two other families of examples: first, Grishanov-Vassiliev formulas \cite{Grishanov}, which are extended to degenerate cases that were excluded in the original paper, and second, given a surface $\Sigma$ we describe a regular invariant that can reasonably be used as a definition of the Whitney index for virtual knots whose projection on $\Sigma$ is a non nullhomotopic curve.

\subsubsection{Grishanov-Vassiliev's planar chain invariants}

\begin{definition}\label{def:nakedArrow}
A \textit{naked arrow diagram} is an arrow diagram with every decoration forgotten except for the local orientations. It is called \textit{planar} if no two of its arrows intersect -- thus one may regard it as a part of the plane, up to isotopy.

A \textit{chain presentation} of such a diagram with $n$ arrows is a way to number its $n+1$ bounded complementary components in the plane from $1$ to $n+1$, in such a way that the numbering increases when one goes from the left to the right of an arrow.

Let $U_n$ be the sum of all planar isotopy equivalence classes of chain presentations of naked arrow diagrams of degree $n$.
$U_n$ is called the \textit{universal degree $n$ planar chain} (\cite{Grishanov}, Definition 1).
\end{definition}

\begin{definition}
\label{def:Conjmap}
An \textit{$h_1$-decorated planar diagram} is the result of assigning an element of $ h_1(\Sigma)\setminus\left\lbrace  1\right\rbrace$ to each region in a planar naked arrow diagram.
\end{definition}

We consider two ways to construct such diagrams:
\begin{enumerate}
\item From the datum of a chain presentation together with a system $\Gamma=\left\lbrace \gamma_1,\ldots,\gamma
_{n+1}\right\rbrace$ -- which yields the notion of $\Gamma$-decorated diagrams.
\item From a planar arrow diagram on $\pi=\pi_1(\Sigma)$: each region of the diagram receives the conjugacy class of the product of the $\pi$-markings at its boundary, in the order induced by the orientation of the circle (the product is not well-defined, but its conjugacy class is).
\end{enumerate}

Call $\Phi_\Gamma$ the sum obtained by decorating every diagram in $U_n$ with a fixed system $\Gamma$ (\cite{Grishanov}, Definition 2). Note that some of the summands in $U_n$ -- namely those with non trivial symmetries -- may lead to the same decorated diagram if some of the $\gamma_i$'s are equal; unlike Grishanov-Vassiliev, we do not forbid that. Of course these summands happen with coefficients greater than $1$ in $\Phi_\Gamma$.

To understand $\Phi_\Gamma$ as an arrow diagram series, we apply the machinery of symmetry-preserving maps from Subsection~\ref{sec:Sinj}. Even though it is absent from the notations, all considered diagrams are planar:

\begin{itemize}
\item $\mathfrak{A}^\Gamma$ is the $\QQ$-space generated by $\Gamma$-decorated planar diagrams (hence of degree $n$). 

\item $\mathfrak{A}^{\rightsquigarrow \Gamma}$ is the subspace of $\mathfrak{A}_n$ generated by planar arrow diagrams on $\pi_1(\Sigma)$ \textit{that induce $\Gamma$-decorated diagrams}.

\item $\mathfrak{A}_n^{1,2,\ldots}$ is the $\QQ$-space generated by all chain presentations of planar naked arrow diagrams of degree $n$. 
\end{itemize}

Note that $\mathfrak{A}_n^{1,2,\ldots}$ is not defined by an equivalence relation on rigid Gauss diagrams on $\pi_1(\Sigma)$. The father of all types of Gauss diagrams here is the type of planar Gauss diagrams on $\pi_1(\Sigma)$ endowed with a chain presentation, such that the chain presentation and the $\pi$-markings induce the same $h_1$-decorated diagram -- it is the pullback of Diagram \eqref{eq:proofGV} below. 

Let us denote by $\S_\Gamma^a$ the symmetry-preserving map $ \widehat{\mathfrak{A}^\Gamma
}
\hookrightarrow \widehat{\mathfrak{A}_n}$, and set $$\widetilde{\Phi}_\Gamma= \S_\Gamma^a(\Phi_\Gamma).$$

\begin{theorem}\label{thm:Grishanov}
For any system of non-trivial conjugacy classes $\Gamma=(\gamma_1,\ldots,\gamma_{n+1})$, $\widetilde{\Phi}_\Gamma$ defines an arrow diagram formula for virtual knots.
\end{theorem}

The reason why $\widetilde{\Phi}_\Gamma$ coincides \textit{as an invariant} with Grishanov-Vassiliev's ${\Phi}_\Gamma$ is Point $2$ of Proposition~\ref{prop:ST}. This theorem improves Theorem $1$ of \cite{Grishanov}, since we remove the assumption that $\Gamma$ is unambiguous (\ie any of the $\gamma_i$'s may coincide) and show that $\widetilde{\Phi}_\Gamma$ is an invariant for virtual knots. 

\begin{proof}

$1.$ First, notice that an equivalence class of arrow diagrams under $\mathrm{AW}$-moves determines an $h_1$-decorated diagram. Thus the map $\S_\Gamma^a$ factorizes through $\S_{aw}^a$ (by point $1.$ of Proposition~\ref{prop:ST}), wherefrom Proposition~\ref{prop:omega} implies that $\widetilde{\Phi}_\Gamma$ defines an invariant under $w$-moves.

$2.$ The fact that no $\gamma_i$ may be trivial gives immediately the condition of $\mathrm{R}_1$ and $\mathrm{R}_2$ invariance from Proposition~\ref{thm:R1et2}.

$3.$ For $\mathrm{R}_3$, the more convenient here is to check condition $3$ of Theorem~\ref{thm:main}. In any $A6T$ relation, only three diagrams possibly have pairwise non intersecting arrows, and either all three of them do, either no one does. This yields two kinds of reduced relations. Let us consider only that on Fig.\ref{pic:reduced6T} (ignore the markings $i$, $j$, $k$ for now), the other case is completely similar. Write the relator of the picture as $A_1-A_2-A_3$ in the order of reading. One has to prove that $$\left\langle  \widetilde{\Phi}_\Gamma,
A_1
\right\rangle=\left\langle  \widetilde{\Phi}_\Gamma,
A_2\right\rangle+\left\langle  \widetilde{\Phi}_\Gamma,
A_3\right\rangle.$$

\begin{figure}[h!]
\centering 
\psfig{file=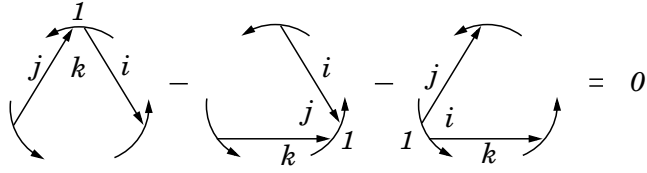,scale=0.60}
\caption{One of the two reduced $6$-term relations for planar diagrams}\label{pic:reduced6T}
\end{figure}

The three spaces defined previously fit into the diagram

\begin{equation} \label{eq:proofGV}
\begin{tikzpicture}[baseline=(current  bounding  box.center)]
\node (1) at (0,0) {$\mathfrak{A}^\Gamma$};
\node (.) at (2.5,0) {.};
\node (2) at (1.5,1.5) {$\mathfrak{A}_n^{1,2,\ldots}$};
\node (3) at (-1.5,1.5) {$\mathfrak{A}^{\rightsquigarrow \Gamma}$};
\draw[->>,>=latex] (2) -- (1) node[midway,below right] {$\T^{1,2,\ldots}_\Gamma$};
\draw[->>,>=latex] (3) -- (1) node[midway,below left] {$\T^a_\Gamma$};
\end{tikzpicture}
\end{equation}

Pick a planar arrow diagram $A$ of degree $n$, and consider the set $\operatorname{Dec}(A)$ of all rigid representatives of all preimages of $\T_\Gamma^{a}(A)$ under the map $\T^{1,2,\ldots}_\Gamma$. By Definition~\ref{def:Sinj}, the cardinality of that set is the sum of all coefficients of $\S_\Gamma^{1,2,\ldots}(\T_\Gamma^{a}(A))$, which means, since the sum of all generators of $\mathfrak{A}_n^{1,2,\ldots}$ is $U_n$, that:

$$\sharp \operatorname{Dec}(A) = \left(U_n,\S_\Gamma^{1,2,\ldots}(\T_\Gamma^{a}(A))\right).$$

But no diagram decorated with a chain presentation admits non-trivial symmetries, whence, by successive applications of Proposition~\ref{prop:ST} $2.$, 

$$\begin{array}{ccl}
\sharp \operatorname{Dec}(A) & = & \left\langle   U_n,\S_\Gamma^{1,2,\ldots}(\T_\Gamma^{a}(A))\right\rangle\\
& = & \left\langle  \Phi_\Gamma,\T_\Gamma^{a}(A)\right\rangle\\
& = & 
\left\langle  \widetilde{\Phi}_\Gamma,A
\right\rangle.
\end{array}$$
Now it remains to see that
$$
\sharp \operatorname{Dec}(A_1)=\sharp \operatorname{Dec}(A_2)+\sharp \operatorname{Dec}(A_3).$$
Look at Fig.\ref{pic:reduced6T} again, now considering the markings $i$, $j$, $k$. Each element of $\operatorname{Dec}(A_1)$ determines either an element of $\operatorname{Dec}(A_2)$, or an element of $\operatorname{Dec}(A_3)$, as indicated by the picture, depending on whether $i<j$ or $i>j$.
This separates $\operatorname{Dec}(A_1)$ into two parts, respectively in bijection with $\operatorname{Dec}(A_2)$ and $\operatorname{Dec}(A_3)$, and terminates the proof.

\end{proof}

\subsubsection{There is a Whitney index for non nullhomotopic virtual knots}

In the classical framework, the Whitney index is an invariant of regular plane curve homotopies which, together with the total writhe number, classifies the representatives of any given knot type up to regular isotopy. In other words, these invariants count the (algebraic) number of Reidemeister I moves that \textit{have to} happen in a sequence of moves connecting two given diagrams. Here we describe such an invariant for \textit{virtual} knot diagrams whose underlying curve on $\Sigma$ is not homotopically trivial.

\subsubsection*{The classical Whitney index}

Let $\delta:\SS^1\rightarrow\RR^2$ be a smooth immersion (with non-vanishing differential). There is an associated Gauss map
$$\begin{array}{cccl}
\Gamma: & \SS^1 & \rightarrow & \SS^1 \\
 & p & \mapsto & u_p(\delta)
\end{array},$$
where $u_p(\delta)$ is the unitary tangent vector to $\delta$ at the point $p$. It depends on a trivialization of the tangent space to $\RR^2$.

\begin{definition}[usual Whitney index]
The index of the above Gauss map only depends on the homotopy class of $\delta$ within the space of smooth immersions. It is called the \textit{rotation number}, or \textit{Whitney index}, of $\delta$.\end{definition}

Given a projection $\RR^3\rightarrow\RR^2$, a generic isotopy of a knot $\SS^1\rightarrow\RR^3$ is called a \textit{regular isotopy} if the corresponding sequence of Reidemeister moves does not involve R-I. Clearly, the Whitney index of planar loops induces an invariant of regular isotopy classes of knots. In practice, it can be easily computed by looking at the Seifert circles of the projection (each of them contributes by $+1$ or $-1$). The total writhe number of a classical knot projection is also an invariant of regular isotopy. These two invariants together satisfy the following

\begin{lemma}[see \citep{KonK}]\label{lem:regularisotopy}
Two equivalent knot diagrams are regularly equivalent if and only if their projections have the same total writhe and the same Whitney index.
\end{lemma}

The proof (which we omit) relies essentially on the \textit{Whitney trick} (see \citep{KonK}) and the following variation table:
\begin{figure}[h!]
\centering 
\psfig{file=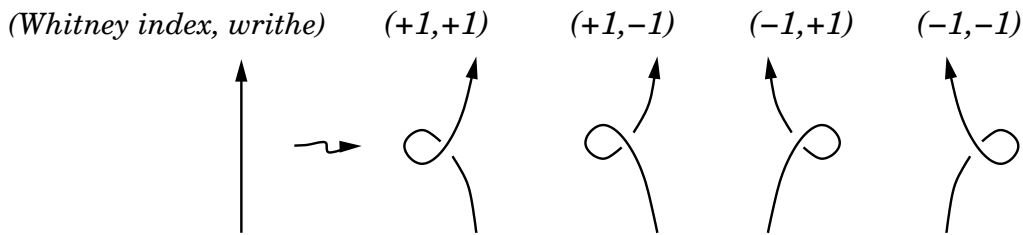,scale=1}
\caption{R-I moves sorted by their effect on the regular invariants}\label{RIsorted}
\end{figure}

\subsubsection*{The virtual framework}

We want to define a Whitney index for virtual knots, that satisfies a version of Lemma~\ref{lem:regularisotopy}. 

Given the virtual Reidemeister I moves, it does not seem reasonable to hope for counting the degree of a Gauss map. Relatedly, the Seifert circles are not embedded any more, and they do not have a well-defined contribution (at least not in the previous sense). Even when one looks only at real knot diagrams, the Whitney index is no more invariant when virtual moves are allowed: see Fig.\ref{pic:regvir}. In other words, though virtual moves do not connect knot diagrams that are not isotopic in the usual sense, they do add bridges between different regular isotopy classes.

\begin{figure}[h!]
\centering 
\psfig{file=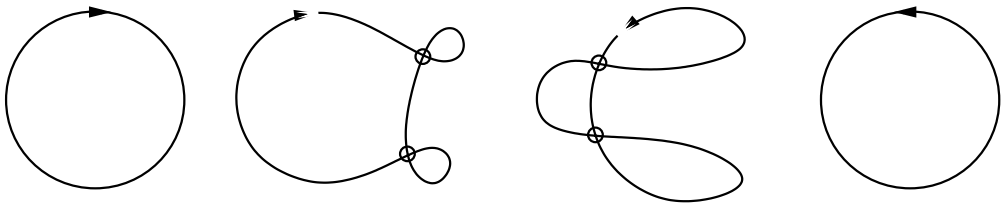,scale=1.2}
\caption{A \enquote{regular} sequence that changes the usual Whitney index}\label{pic:regvir}
\end{figure}

From now on, let $\Sigma$ be an orientable surface with non trivial fundamental group.

\begin{definition}\label{def:regvirtuals}
Two virtual knot diagrams in $\Sigma$ are called \textit{regularly equivalent} if they are connected by a sequence of moves that does not involve the \textit{real} Reidemeister I move.
\end{definition}

In \citep{KauffmanVKT99}, regular equivalence is defined as equivalence under all moves but the real R-I \textit{and} the virtual R-I too. Here, our goal is to define maps on the set of Gauss diagrams, so implicitly they must be invariant under all detour moves anyway.

\begin{figure}[h!]
\centering 
\psfig{file=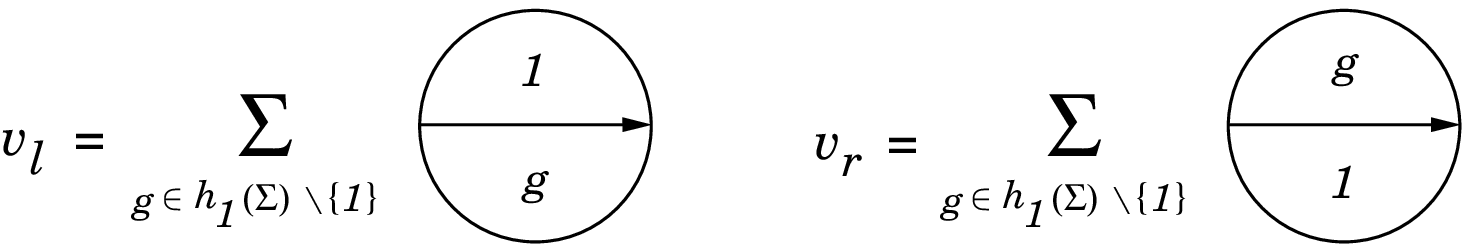,scale=0.85}
\caption{The regular invariants for non nullhomotopic virtual knots}\label{pic:regvirtual}
\end{figure}

\begin{lemma}\label{lem:reginsigma}
The $h_1$-decorated diagram series $v_l$ and $v_r$ from Fig.\ref{pic:regvirtual} define invariants of regular equivalence. Moreover, two virtual knot diagrams from the same knot type with non trivial homotopy class are regularly equivalent as soon as both $v_l$ and $v_r$ coincide on them.
\end{lemma}

\begin{proof}
\textit{First assertion.} First, notice that since $\Sigma$ is orientable, $w=w_1(T\Sigma)$ is trivial and the $w$-moves never change the orientation of an arrow. Hence a $w$-orbit of planar Gauss diagrams on $\pi_1(\Sigma)$ determines an $h_1(\Sigma)$-decorated diagram (see Definition~\ref{def:Conjmap}). In restriction to the diagrams that happen in $v_l$ and $v_r$, this forgetful map is actually a $1-1$ correspondence. 
Hence $v_l$ and $v_r$ can be regarded as series of $w$-orbits of Gauss diagrams, so that they satisfy the $w$ invariance criterion (Proposition~\ref{prop:omega}). The invariance under $\mathrm{R}_2$ and $\mathrm{R}_3$ moves follows from Theorem \ref{thm:mainorbits}.

\textit{Second assertion.} The proof is essentially the same as that of Lemma~\ref{lem:regularisotopy}. A little loop can run along a knot diagram without using R-I moves even where there are virtual crossings.  
The table from Fig.\ref{RIsorted} becomes, respectively (for the couple ($v_l,v_r$)):
$$\begin{array}{cccc}
(0,+1) & (-1,0) & (+1,0) &(0,-1).
\end{array}$$
This is the essential reason for which one needs the assumption that the knot diagrams are not nullhomotopic: otherwise the increase would be $(0,0)$ for all R-I moves.
\end{proof}

Finally, looking at the above table, which details \textit{how} $v_l$ and $v_r$ control the R-I moves, it appears that:

\textbf{1.} $v_r+v_l$ behaves like the total writhe number under Reidemeister moves: it has the same \enquote{derivative}. It follows that these differ by a constant that depends only on the virtual knot type, \ie a virtual knot invariant in the usual sense (see Fig.\ref{pic:totwrithe}).

\begin{figure}[h!]
\centering 
\psfig{file=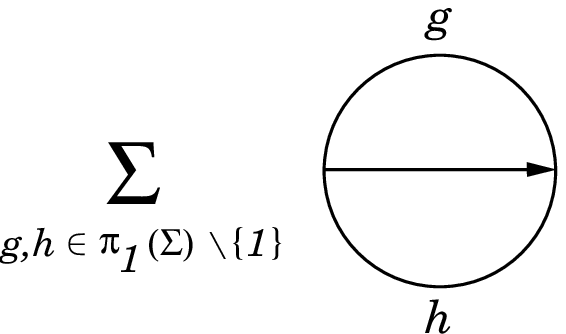,
scale=0.87}
\caption{The difference between $v_r+v_l$ and the (usual) total writhe number}\label{pic:totwrithe}
\end{figure}

\textbf{2.} \textit{In restriction to real knot diagrams}, $v_r-v_l$ has the same derivative as the Whitney index, with respect to Reidemeister moves. Hence there is an invariant of real knots $c$ such that for every real knot diagram $D$, $v_r(D)-v_l(D)+c(D)$ is the usual Whitney index of $D$.

Hence we have proved:

\begin{deflemma}Call $v_r-v_l$ the \emph{Whitney index} of non nullhomotopic virtual knot diagrams, and call $v_r+v_l$ their \emph{writhe number}. These make Lemma~\ref{lem:regularisotopy} hold for non nullhomotopic virtual knot diagrams.
\end{deflemma}

\bibliographystyle{plain}
\bibliography{bibli}

\end{document}